\documentclass[a4paper,10pt]{article}

\usepackage{geometry}
\usepackage{amsmath}
\usepackage{amsfonts,color,amssymb,mathrsfs,stmaryrd}
\usepackage{amsthm}
\usepackage{amssymb}
\usepackage{url}
\usepackage{bera}
\usepackage[normalem]{ulem}
\usepackage{graphicx}
\usepackage{bbm}
\usepackage{enumitem}
\usepackage{tikz}
\usepackage{hyperref}

\numberwithin{equation}{section}

\newtheorem{theorem}{Theorem}\numberwithin{theorem}{section}

\newtheorem{prop}[theorem]{Proposition}
\newtheorem{lemma}[theorem]{Lemma}
\newtheorem{example}[theorem]{Example}

\newtheorem{rema}[theorem]{Remark}

\newcommand{\p}{\mathbb{P}}
\newcommand{\E}{\mathbb{E}}
\newcommand{\Q}{\mathbb{Q}}
\newcommand{\MM}{\mathbb{M}}
\newcommand{\supp}{{\rm supp}}
\newcommand{\Var}{{\rm Var}}

\usepackage{titlesec}
\titleformat*{\section}{\normalfont\large\bfseries}
\titleformat*{\subsection}{\normalfont\bfseries}

\date{\vspace{-0.95cm}}

\begin{document}
	\title{\bfseries Geometric functionals of polyconvex excursion sets of Poisson shot noise processes}
	
	\author{%
		Vanessa Trapp\footnotemark[1]%
	}
	
	\date{}
	\renewcommand{\thefootnote}{\fnsymbol{footnote}}

	\footnotetext[1]{%
		Hamburg University
		of Technology, Germany. Email: vanessa.trapp@tuhh.de
	}

	\maketitle
	\begin{abstract}\noindent
		Excursion sets of Poisson shot noise processes are a prominent class of random sets. We consider a specific class of Poisson shot noise processes whose excursion sets within compact convex observation windows are almost surely polyconvex. This class contains, for example, the Boolean model. In this paper, we analyse the behaviour of geometric functionals such as the intrinsic volumes of these excursion sets for growing observation windows. In particular, we study the asymptotics of the expectation and the variance, derive a lower variance bound and show a central limit theorem.
		
		
	\end{abstract}
	\section{Introduction and model}
	\subsection{General introduction}
	Random fields and their excursion sets are an important topic of probability theory with a wide range of applications. They are used in various fields for modelling random signals, with excursion sets corresponding to the regions where these signals interfere and exceed some given threshold.
	In 1944, for example, random fields were used in \cite{R44} to model random noise in electronic devices and analyse shot effects, which occur if a lot of disturbances overlap. Beyond this, random fields have a variety of applications in different areas such as wireless communication networks (see e.g. \cite{BB01,BB09}), insurance mathematics (see for instance \cite{B00,KM95,T04}) and medicine (see e.g.\ \cite{TW07,WMNVFE96}), to name just a few.
	For a general introduction to random fields we refer the reader to \cite{AT07,AW09}.
	
	While the best studied underlying fields in the literature are Gaussian random fields, we focus in this paper on another important choice of underlying fields, namely Poisson shot noise processes. 
	Let $(\MM,\mathcal{F}_\MM,\Q)$ be a mark space with $\sigma$-field $\mathcal{F}_\MM$ and probability measure $\Q$. For $d\in\mathbb{N}$ we consider a marked Poisson process $\eta$ on $\hat{\mathbb{R}}^d=\mathbb{R}^d\times\mathbb{M}$ with intensity measure $\lambda=\gamma\lambda_d\otimes\mathbb{Q}$ for some $\gamma>0$, where $\lambda_d$ denotes the $d$-dimensional Lebesgue measure. For a family of non-negative measurable kernel functions $(g_m)_{m\in\MM}$ with $g_m:\mathbb{R}^d\to [0,\infty)$, a Poisson shot noise process $(f_\eta(y))_{y\in\mathbb{R}^d}$ is defined by 
	\begin{align*}
		f_\eta(y)=\sum_{(x,m)\in\eta}g_m(y-x)
	\end{align*}
	for $y\in \mathbb{R}^d$. For a fixed $u>0$, its excursion set is given by
	\begin{align*}
		Z_u=\{y\in \mathbb{R}^d:f_\eta(y)\geq u\}.
	\end{align*}
	We usually consider excursion sets within compact convex observation windows for fixed $u>0$ and want to study their behaviour for a growing sequence of observation windows. This means, we consider an increasing sequence of windows $(W_r)_{r\geq 1}$ with $W_r=rW$ for $r\geq 1$ and some compact convex set $W$ with $\lambda_d(W)>0$ and are interested in the asymptotic behaviour of $\varphi(Z_u\cap W_r)$, where $\varphi$ is a functional, which provides geometric information about $Z_u\cap W_r$.
	
	In the literature there already exist a few results for specific geometric functionals of excursion sets of different random fields. For Gaussian random fields, \cite{EL16} establishes a central limit theorem for the Euler characteristic of excursion sets, while \cite{BBDE19,BE17} use the Euler characteristic to construct tests for Gaussianity. Central limit theorems for Lipschitz-Killing curvatures of Gaussian excursion sets are proven in \cite{KV18,M17}.
	More general random fields are considered for instance in \cite{AST10, BD20, BBDE19, BST12, S13}. Among other things, the Euler characteristic, surface area and volume of excursion sets of two-dimensional stationary isotropic random fields are analysed in \cite{BBDE19}.
	In \cite{AST10, BD20} the asymptotic expectation of functionals like Lipschitz-Killing curvatures, the level perimeter or level total curvature integrals are calculated for different fields such as stable fields or  two-dimensional smooth stationary fields. An overview on asymptotic results for the volume of stationary random fields is given in \cite{S13}. In \cite{BST12} the volume of excursion sets of quasi-associated random fields are considered and central limit theorems are shown, in particular for specific Gaussian fields and Poisson shot noise processes.
	Further research on the asymptotic behaviour of the volume of excursion sets of Poisson shot noise processes is for example done in \cite{L19,LPY22,ST22}, which show limit theorems or variance asymptotics under different conditions on the kernel functions. Moreover, the perimeter of Poisson shot noise processes is considered in \cite{BD16,L19,LPY22}. While \cite{BD16} analyses its expectation, the works \cite{L19,LPY22} show central limit theorems for the perimeter. Additionally, \cite{L19} shows asymptotic results for other functionals like the Euler characteristic.

	In this paper we want to study general geometric functionals of excursion sets of Poisson shot noise processes.
	More precisely, we assume that the functional $\varphi$ is measurable, translation invariant, additive and locally bounded. The most prominent example of geometric functionals are the intrinsic volumes $V_i$ for $i\in\{0,\dots,d\}$, where for a $d$-dimensional compact convex set $K\subseteq\mathbb{R}^d$, $V_d(K)$ is its volume, $V_{d-1}(K)$ is half the surface area and $V_0(K)$ is the Euler characteristic of $K$ (see e.g.\ \cite[Chapter 4]{S14}). In addition to the intrinsic volumes there are also other types of geometric functionals like the more general example of mixed volumes $V(K[j],K_1,\dots,K_{d-j})$, where $K[j]$ means that $K$ is repeated $j$ times, $j\in\{1,\dots,d\}$ and $K_1,\dots,K_{d-j}$ are fixed compact convex sets (see \cite[Section 5.1]{S14}), or total measures from translative integral geometry (see \cite[Section 6.4]{SW08}). For further details and more examples of geometric functionals see also \cite[p. 79]{HLS16} and the references therein.
	
	Unlike functionals like the volume or the Euler characteristic, general geometric functionals are only well-defined for polyconvex sets. This requires the introduction of a model, i.e.\ conditions on the family of kernel functions, for which the corresponding excursion sets in the observation window are almost surely polyconvex. In detail, our model is introduced in Section \ref{section1.2}.
	A special case of our model is the Boolean model. 
	The underlying proof strategy for Poisson shot noise processes is an extended version of the one for the Boolean model (from e.g.\ \cite{HLS16}). However, the case of Poisson shot noise processes is substantially more intricate and its proofs rely heavily on novel ideas since, in contrast to the Boolean model, excursion sets of Poisson shot noise processes are in general not standard random sets (see also Section \ref{sec:boolean_model} and Example \ref{example:notstandard}).
	
	In the remaining part of this section we introduce our a model and compare it with the Boolean model in more detail. Section $2$ is devoted to geometric and stochastic preliminaries. In Section $3$ the main results, i.e.\ the results for the asymptotic behaviour of expectation and variance and qualitative and quantitative central limit theorems, are presented. The proofs of these results can be found in Section $4$. Finally, in the appendix we prove a generalised version of the reverse Poincar\'e inequality from \cite[Theorem 1.1]{ST22}, which we use to show a lower variance bound.
	
	\subsection{A model with polyconvex excursion sets}
	\label{section1.2}
	In this section we introduce conditions for the kernel functions which guarantee that excursion sets in the compact convex observation window are almost surely polyconvex, i.e.\ that $Z_u\cap W_r$ can almost surely be written as a union of finitely many compact convex sets. Then, $\varphi(Z_u\cap W_r)$ is well-defined.
	
	By $\mathcal{K}^d$ we denote the set of compact convex sets in $\mathbb{R}^d$. The space of closed sets $\mathcal{F}^d$ in $\mathbb{R}^d$ is equipped in this paper with the Fell topology whose $\sigma$-algebra $\mathcal{B}(\mathcal{F}^d)$ can be generated by $\{\{F\in\mathcal{F}^d: F\cap U\neq\emptyset\}: U\in \mathcal{U}^d\}$, where $\mathcal{U}^d$ denotes the set of open sets in $\mathbb{R}^d$ (see e.g.\ \cite[Lemma 2.1.1]{SW08}). By $\widetilde{L}^1(\mathbb{R}^d)$ we denote the space of integrable functions on $\mathbb{R}^d$ with compact convex closed support $K$, whose restriction to $K$ is continuous. This space will be equipped with the Borel $\sigma$-algebra $\mathcal{L}_1$, which is induced by the $L^1$-norm.
	We assume in this paper that $(g_m)_{m\in\MM}$ is a family of measurable functions $\widetilde{L}^1(\mathbb{R}^d)\ni g_m:\mathbb{R}^d\to[0,\infty)$, which are concave on their non-empty compact convex closed supports $K_m\in\mathcal{K}^d$ with $\lambda_d(K_m)>0$ for $m\in\MM$.  Note that $K_m=\overline{\supp(g_m)}=\overline{\{x\in\mathbb{R}^d:g_m(x)>0\}},$ where $\bar{A}$ denotes the closure of a set $A\subset\mathbb{R}^d$. 
	The continuity at inner points of $K_m$ also follows from the concavity on $K_m$ so that the continuity condition describes the behaviour of the function at the boundary of $K_m$.
	Moreover, we assume that $m\mapsto g_m$ is $\mathcal{F}_{\MM}$-$\mathcal{L}_1$-measurable. Note that $m\mapsto K_m$ is then automatically $\mathcal{F}_{\MM}$-$\mathcal{B}(\mathcal{F}^d)$-measurable since for any $U\in\mathcal{U}^d$, $\{f\in \widetilde{L}^1(\mathbb{R}^d):\overline{\supp(f)}\cap U\neq\emptyset\}$ is an open set and therefore contained in $\mathcal{L}_1$. Hence, $g_m\mapsto K_m$ is $\mathcal{L}_1$-$\mathcal{B}(\mathcal{F}^d)$-measurable. Together with the assumption that $m\mapsto g_m$ is $\mathcal{F}_{\MM}$-$\mathcal{L}_1$-measurable, this provides that $m\mapsto K_m$ is $\mathcal{F}_{\MM}$-$\mathcal{B}(\mathcal{F}^d)$-measurable.
	Additionally, we assume that 
	\begin{align}\label{Assumption first moment}
		\int_\MM V_i(K_{m})\;\mathbb{Q}(\mathrm{d}m)<\infty
	\end{align}
	for $i\in\{1,\dots,d\}$.
	In the course of this paper we also assume higher order moment conditions of the form
	\begin{align}\label{Assumption_fourth_moment}
		\int_\MM V_i(K_m)^k\;\Q(\mathrm{d}m)<\infty
	\end{align}
	of up to order $k=4$ for $i\in\{1,\dots,d\}$ to derive variance asymptotics and central limit theorems.
	
	We often represent a point $\hat{x}\in\eta$ as $\hat{x}=(x,m)\in\mathbb{R}^d\times\MM$ and abbreviate 
	\begin{align*}
		K(\hat{x})=K_m+x=\{y\in\mathbb{R}^d:y-x\in K_{m}\}.
	\end{align*}
	Let  $K_{m}+W_r=\{x+y:x\in K_{m},y\in W_r\}$ denote the Minkowski sum of the sets $K_{m}$ and $W_r$ and define 
	\begin{align*}
		S_{W_r}=\{\hat{x}=(x,m)\in\mathbb{R}^d\times\MM: K(\hat{x})\cap W_r\neq \emptyset\}.
	\end{align*}
	\noindent Since 
	\begin{align*}
		V_d(\{x\in\mathbb{R}^d:K(\hat{x})\cap W_r\neq\emptyset\})=V_d(\{x\in\mathbb{R}^d:(K_{m}+x)\cap W_r\neq\emptyset\})
		= V_d(-K_{m}+W_r)
	\end{align*}
	and we can bound $V_d(-K_{m}+W_r)$ with Steiner's formula (see also \eqref{eq:steiner}) by a linear combination of the intrinsic volumes of $K_{m}$, where the coefficients depend on $W_r$, the moment assumption in \eqref{Assumption first moment} provides that $\E[\eta(S_{W_r})]\leq\gamma\int_\MM V_d(-K_{m}+W_r)\;\mathbb{Q}(\mathrm{d}m)<\infty$ for any $r\geq1$ and thus, $\eta(S_{W_r})<\infty$ almost surely. 
	Let now $n=\eta(S_{W_r})$ and denote the points of the Poisson process in $S_{W_r}$ by $\hat{x}_1=(x_1,m_1),\dots,
	\hat{x}_n=(x_n,m_n)$. For $j\in\{1,\dots,n\}$ we often use the short notation
	\begin{align*}
		\hat{K}_j=K(\hat{x}_j)=K_{m_j}+x_j.
	\end{align*}
	For $\emptyset\neq I\subseteq\{1,\dots,n\}$ we define
	\begin{align}\label{eq_XI}
		X_I=\Big\{y\in W_r:\sum_{j\in I}g_{m_j}(y-x_j)\geq u, y\in\hat{K}_j \text{ for all } j\in I\Big\}.
	\end{align}
	We use these sets to show that $Z_u\cap W_r$ is almost surely polyconvex.
	\begin{prop}
		\label{lemma E polyconvex}
		Let $u>0$ be fixed, $n=\eta(S_{W_r})<\infty$ a.s.\ and $X_I$ be defined as in \eqref{eq_XI} for $\emptyset\neq I\subseteq\{1,\dots,n\}$. Then, $X_I$ is a compact convex set for all $\emptyset\neq I\subseteq\{1,\dots,n\}$ and
		\begin{align*}
			Z_u\cap W_r=\bigcup\limits_{\emptyset\neq I\subseteq\{1,\dots,n\}}X_I
		\end{align*}
		for all $r\geq1$.
	\end{prop}
	\begin{proof}
		In the following we show that the set $X_I$ for $\emptyset\neq I\subseteq\{1,\dots,n\}$ is a compact convex set and that we can write $Z_u\cap W_r$ as a union of these sets.		
		For $y_1,y_2\in X_I$ and $\alpha\in(0,1)$ we have $\alpha y_1+(1-\alpha)y_2\in\hat{K}_j$ since $\hat{K}_j$ is convex for all $j\in I$. Due to the concavity of $g_{m_j}$ on $K_{m_j}$ it holds
		\begin{align*}
			\sum_{j\in I}g_{m_j}(\alpha y_1+(1-\alpha)y_2-x_j)\geq \sum_{j\in I}\alpha g_{m_j}(y_1-x_j)+(1-\alpha)g_{m_j}(y_2-x_j)\geq u.
		\end{align*}
		Hence, $X_I$ is convex for $\emptyset\neq I\subseteq\{1,\dots,n\}$. Since $g_{m_j}\in\widetilde{L}^1(\mathbb{R}^d)$, the restriction of $g_{m_j}$ to $K_{m_j}$ is continuous for all $j\in \{1,\dots,n\}$. Together with the fact that $W_r$ is a compact convex set and $X_I\subseteq W_r$ for all $r\geq1$, this ensures that $X_I$ is compact for $\emptyset\neq I\subseteq\{1,\dots,n\}$.
		
		Moreover, $\bigcup_{\emptyset\neq I\subseteq\{1,\dots,n\}}X_I\subseteq Z_u\cap W_r$ since $X_I\subseteq Z_u\cap W_r$ for all $I\subseteq\{1,\dots,n\}$. As there is also $\emptyset\neq I\subseteq\{1,\dots,n\}$ with $y\in X_I$ for all $y\in Z_u\cap W_r$, it follows
		\begin{align*}
			Z_u\cap W_r=\bigcup\limits_{\emptyset\neq I\subseteq\{1,\dots,n\}}X_I,
		\end{align*}
		which completes the proof.
	\end{proof}
	Because of Proposition \ref{lemma E polyconvex} the geometric functional of $Z_u\cap W_r$ is defined via the inclusion exclusion principle (see also Equation \eqref{eq:intrinsic volume polyconvex set}). We have for $n=\eta(S_{W_r})$,
	\begin{align*}
		\varphi(Z_u\cap W_r)=\sum_{J\subseteq\mathcal{P}(\{1,\dots,n\})\backslash\emptyset:\lvert J\rvert\geq 1 }(-1)^{\lvert J\rvert+1}\varphi(\cap_{I\in J}X_{I}),
	\end{align*}
	where $\mathcal{P}(\{1,\dots,n\})$ denotes the power set of $\{1,\dots,n\}$.
	\subsection{Comparison to the Boolean model}
	\label{sec:boolean_model}
	An important example of random sets are Boolean models, which correspond to considering the union of the supports $Z=\bigcup_{\hat{x}\in\eta}K(\hat{x})$. The behaviour of $\varphi(Z\cap W_r)$ for the Boolean model $Z$ was studied, for example, in \cite{HLS16,LP18,SW08}. It can be shown that the expectation grows with order $V_d(W_r)$ (see e.g.\ \cite[Chapter 9.1]{SW08}). Variance asymptotics of order $V_d(W_r)$ and corresponding quantitative central limit theorems for the Boolean model were derived in \cite[Chapter 22.2]{LP18} or \cite[Sections 3 and 9]{HLS16}. One way to generalise the Boolean model is to consider the union of cylinders of Poisson cylinder processes, which have long-range dependencies. Variance asymptotics and central limit theorems for the case of Poisson cylinder processes were shown in \cite[Theorem 3.5 and Theorem 3.10]{BST22}.

	In this paper we want to study Poisson shot noise processes, which also generalise the Boolean model. In fact, the Boolean model arises as a special case of our model if $g_m(x)\geq u$ for all $x\in K_m$ and $m\in\MM$.  Since we do not consider the union of the supports but the union of the sets $X_I$ for $\emptyset\neq I\subseteq\{1,\dots, \eta(S_{W_r})\}$, in general our random sets and the Boolean model differ mainly in two aspects in which Boolean models and unions of the cylinders of Poisson cylinder processes do not deviate. First, $X_I$ depends on $\lvert I \rvert$ points while $K(\hat{x})$ only depends on the point $\hat{x}$. This also implies that two sets $X_I, X_J$ can depend on the same point for $I\neq J$ if $I\cap J\neq \emptyset$, while $K(\hat{x})$ and $K(\hat{y})$ cannot depend on the same point if $\hat{x}\neq\hat{y}$. 
	The second main difference arises from the possible number of convex sets needed to represent the excursion set as a union of these convex sets. Let $N(A)$ denote the smallest number needed to write a polyconvex set $A$ as a union of $N(A)$ compact convex sets. Then, for the Boolean model $Z$, $N(Z\cap C^d)\leq \eta(S_{C^d})$ for $C^d=[0,1]^d$. For our model we can in general need up to $2^{\eta(S_{C^d})}-1$ sets (cf. Proposition \ref{lemma E polyconvex}). This means that $\E[2^{N(Z_u\cap C^d)}]$ is not necessarily finite (see also Example \ref{example:notstandard}) and therefore, $Z_u$ is in contrast to $Z$ in general not a standard random set as defined in \cite[Definition 9.2.1]{SW08}.

	\section{Preliminaries}
	
	\subsection{Geometric functionals and intrinsic volumes}
	\label{section:intrinsic volumes}
	Let $\mathcal{R}^d$ denote the convex ring, which contains all subsets of $\mathbb{R}^d$ that can be written as a finite union of compact convex sets. 
	In this paper we consider geometric functionals, i.e.\ $\mathcal{B}(\mathcal{F}^d)$-$\mathcal{B}(\mathbb{R})$-measurable functions $\varphi:\mathcal{R}^d\to\mathbb{R}$, which are translation invariant, additive and locally bounded. Translation invariance means that $\varphi(A+x)=\varphi(A)$ for $A\in\mathcal{R}^d$ and $x\in\mathbb{R}^d$, where $A+x=\{y+x, y\in A\}$. Additivity is used to describe that $\varphi(\emptyset)=0$ and $\varphi(A_1\cup A_2)=\varphi(A_1)+\varphi(A_2)-\varphi(A_1\cap A_2)$ for $A_1,A_2\in\mathcal{R}^d$. Local boundedness means that $\lvert\varphi(A)\rvert$ is uniformly bounded, i.e.
	\begin{align}\label{eq:locallybounded}
		\lvert\varphi(A)\rvert\leq M_\varphi
	\end{align}
	for all $\mathcal{K}^d\ni A\subseteq [0,1]^d$, where $M_\varphi$ only depends on $\varphi$. Together with the translation invariance, this means that this inequality also holds for each translation of $[0,1]^d$.
	For $A=\bigcup_{j=1}^nK_j\in\mathcal{R}^d$ for $K_j\in\mathcal{K}^d$, $j\in\{1,\dots,n\}$ and $n\in\mathbb{N}$, additivity opens up the possibility of tracing the definition of $\varphi(A)$ back to $\varphi(K_j)$, $j\in\{1,\dots,n\}$. The inclusion exclusion formula  provides
	\begin{align}\label{eq:intrinsic volume polyconvex set}
		\varphi(A)=\sum_{\emptyset\neq J\subseteq\{1,\dots,n\}}(-1)^{\lvert J\rvert+1}\varphi\Big(\bigcap_{j\in J}K_j\Big).
	\end{align}
	In harmony with this we define $\varphi(A\cap C_0^d)$ for $A\in\mathcal{R}^d$ and the half open unit cube $C_0^d=[0,1)^d$ as 
	\begin{align}
		\label{eq:c0^d}
		\varphi(A\cap C_0^d)=\varphi(A\cap C^d)-\varphi(A\cap \partial^+C^d),
	\end{align}
	where $C^d=[0,1]^d$ and $\partial^+C^d=C^d\backslash C_0^d$ (see also \cite[p. 394]{SW08}).
	
	The parallel body $K^r$ of a compact convex body $K\in\mathcal{K}^d$ is defined by
	\begin{align*}
		K^r=K+B^d(\mathbf{0},r)=\{a+b:a\in K,b\in B^d(\mathbf{0},r)\}=\{x\in\mathbb{R}^d:\mathrm{d}(K,x)\leq r\}
	\end{align*}
	for $r\geq 0$, where $B^d(\mathbf{0},r)$ denotes the $d$-dimensional closed ball around the origin $\mathbf{0}=(0,\dots,0)\in\mathbb{R}^d$ of radius $r$. With this definition we can introduce intrinsic volumes, which are important examples of geometric functionals, as they correspond to geometric properties like volume, surface area or Euler characteristic.
	By the Steiner formula we know that the volume of a parallel body is a polynomial of degree at most $d$ in the radius of the ball. It is given by
	\begin{align}\label{eq:steiner}
		V_d(K^r)=\sum_{k=0}^{d}r^{d-k}\kappa_{d-k}V_k(K)\leq C_1\sum_{k=0}^{d}V_k(K),
	\end{align}
	where $\kappa_{d-k}$ denotes the volume of the $(d-k)$-dimensional unit ball (see for example \cite[p. 2]{SW08}) and $C_1>0$ is a suitable constant depending on $d$ and $r$.
	One of the basic properties of intrinsic volumes, which we use throughout this paper, is their monotonicity and positivity on $\mathcal{K}^d$, i.e.\ for $K_1,K_2\in\mathcal{K}^d$ with $K_1\subseteq K_2$ it holds $V_i(K_1)\leq V_i(K_2)$ and $V_i(K_1)\geq 0$ (see for instance \cite[Remark 3.22]{HW20}). Moreover, we use their homogeneity $V_i(rK)=r^iV_i(K)$ for all $r\geq0$ and $K\in \mathcal{K}^d$ (see e.g.\ \cite[Remark 3.21]{HW20}). 
	For the volume $V_d$ we know from e.g.\ \cite[Theorem 5.2.1]{SW08} the translative integral formula
	\begin{align}\label{eq:int_V_d A_1+xcapA_2}
		\int V_d((K_1+x)\cap K_2)\;\mathrm{d}x=V_d(K_1)V_d(K_2)
	\end{align}
	for $K_1,K_2\in\mathcal{K}^d$, where $\mathrm{d}x$ stands for the integration with respect to the $d$-dimensional Lebesgue measure.
	The next lemma summarises three further properties of intrinsic volumes from \cite[Lemma 3.6 and Lemma 3.7]{HLS16} and \cite[p. 613]{SW08}, which we need for the proof of our main theorems.
	\begin{lemma}
		Let $K_1,K_2\in\mathcal{K}^d$. Then it holds for suitable constants $C_2,C_3>0$, which might depend on $d$,
		\begin{align}\label{eq:lambda_A_1+xcappartialA_2}
			V_d(\{x\in\mathbb{R}^d:(K_1+x)\cap\partial K_2\neq\emptyset\})\leq C_2\sum_{k=0}^{d-1}V_k(K_2)\sum_{\ell=0}^dV_{\ell}(K_1)
		\end{align}	
		and for $i\in\{0,\dots,d\}$,
		\begin{align}\label{eq:frationV_iV_d}
			\frac{V_i(K_1)}{V_d(K_1)}\leq C_3\frac{1}{r(K_1)^{d-i}}
		\end{align}
		for $r(K_1)>0$, where $r(K_1)$ denotes the inradius of $K_1$. Moreover, for $i,j\in\{0,\dots,d\}$ with  $i<j$ there exist constants $C(i,j)>0$, which may depend on $i,j$ and $d$, such that
		\begin{align}\label{eq:ViVj}
			V_j(K_1)\leq C(i,j) V_i(K_1)^{j/i}.
		\end{align}
	\end{lemma}
	Analogous to \cite[Lemma 3.5]{HLS16} one can show the following lemma.
	\begin{lemma}
		There exists a constant $C_4>0$, depending on  $d$, $\gamma$ and the first moment in \eqref{Assumption first moment}, such that
		\begin{align}\label{eq:int_Acapk1capkn}
			\int \sum_{k=0}^d V_k(K\cap \hat{K}_1\cap\hdots\cap \hat{K}_n)\;\lambda^n(\mathrm{d}(\hat{x}_1,\dots,\hat{x}_n))\leq C_4^n\sum_{k=0}^{d}V_k(K)
		\end{align}
		for all $K\in\mathcal{K}^d$ and $n\in\mathbb{N}$.
	\end{lemma}
	\subsection{Normal approximation of Poisson functionals}
	Let $(\mathbb{X},\mathcal{X})$ be a measurable space with a $\sigma$-finite intensity measure $\mu$ and denote by $\mathbf{N}$ the set of $\sigma$-finite counting measures supplied with the $\sigma$-field $\mathcal{N}$ generated by the mappings of the form $\nu\mapsto\nu(B)$ for $B\in\mathcal{X}$. 
	A Poisson process $\xi$ on $(\mathbb{X},\mathcal{X})$ with intensity measure $\mu$ is a random element of $(\mathbf{N},\mathcal{N})$ which satisfies
	\begin{enumerate}[label=(\roman*)]
		\item $\xi(B)$ is Poisson distributed with parameter $\mu(B)$ for all $B\in\mathcal{X}$,
		\item $\xi(B_1),\dots,\xi(B_n)$ are independent for pairwise disjoint sets $B_1,\dots, B_n\in\mathcal{X}$ and $n\in\mathbb{N}$.
	\end{enumerate}
	A Poisson functional $F$ is a random variable that depends on the Poisson process $\xi$, i.e.\ $F=f(\xi)$, where $f$ is a real-valued measurable function on $\mathbf{N}$ and is called representative. For simplicity we use the notation $F=F(\xi)$. If $F$ is square-integrable, we write $F\in L^2_\xi$.
	
	The difference operator of a Poisson functional $F$ is defined by
	\begin{align*}
		D_{x}F=F(\xi+\delta_{x})-F(\xi)
	\end{align*}
	for $x\in\mathbb{X}$, where $\delta_x$ denotes the Dirac measure concentrated at $x$. Iteratively, one can define higher order difference operators as
	\begin{align*}
		D_{x_1,\dots,x_n}^nF=D_{x_1}(D_{x_2,\dots,x_n}^{n-1}F)
	\end{align*}
	for $n\in\mathbb{N}$ and $x_1,\dots,x_n\in\mathbb{X}$.
	With the help of these difference operators we can introduce the Fock space representation, which expresses the variance in terms of difference operators. By \cite[Theorem 18.6]{LP18} we have for $F\in L^2_\xi$,
	\begin{align}\label{eq:fockspace}
		\Var[F]=\sum_{n=1}^{\infty}\frac{1}{n!}\int(\E[D_{x_1,\dots,x_n}^nF])^2\;\mu^n(\mathrm{d}(x_1,\dots,x_n)).
	\end{align}
	This representation will be used to calculate the asymptotic variance of $\varphi(Z_u\cap W_r)$. To show the positivity of the asymptotic variance, we use the following theorem, which is a generalised version of the reverse Poincar\'e inequality derived in \cite[Theorem 1.1]{ST22}. A proof of this new lower variance bound can be found in Appendix \ref{app:lower_variance}.
	\begin{theorem}\label{thm:lowervarbound}
		Let $F\in L^2_\xi$ satisfy
		\begin{align}\label{Assumption_variance_theoretisch}
			\E\int(D_{x_1,\dots,x_{k+1}}^{k+1}F)^2\mu^{k+1}(\mathrm{d}(x_1,\dots,x_{k+1}))\leq \alpha\E\int(D_{x_1,\dots,x_{k}}^{k}F)^2\;\mu^k(\mathrm{d}(x_1,\dots,x_k))<\infty
		\end{align}
		for some $\alpha\geq 0$ and $k\in\mathbb{N}$. Then,
		\begin{align*}
			\Var[F]\geq \frac{1}{c(\alpha,k)}\E\int(D_{x_1,\dots,x_{k}}^{k}F)^2\;\mu^k(\mathrm{d}(x_1,\dots,x_k))
		\end{align*}
		for some constant $c(\alpha,k)>0$, which depends only on $\alpha$ and $k$.
	\end{theorem}
	
	Beside expectation and variance we also derive quantitative central limit theorems in terms of the Wasserstein and the Kolmogorov distance. For a Poisson functional $F$ with $\E[\lvert F\rvert]<\infty$ the Wasserstein distance of $F$ to a standard Gaussian random variable $N$ is given by
	\begin{align*}
		\mathrm{d}_W(F,N)=\sup_{h\in \mathrm{Lip}_1}\lvert\E[h(F)]-\E[h(N)]\rvert,
	\end{align*}
	where $\mathrm{Lip}_1$ denotes the set of Lipschitz functions with Lipschitz constant less than or equal to one. The Kolmogorov distance of these random variables is defined by
	\begin{align*}
		\mathrm{d}_K(F,N)=\sup_{t\in\mathbb{R}}\;\lvert\p(F\leq t)-\p(N\leq t)\rvert.
	\end{align*} 
	To bound the Wasserstein distance, we use the following bounds from \cite[Theorem 1.1]{LPS16} and \cite[Theorem 3.1]{BPT20}.
	\begin{theorem}
		\label{thm:clt_theory_Wasserstein}
		Let $F\in L_\xi^2$ be a Poisson functional with $\E[F]=0$, $\Var[F]=1$ and $\E\int(D_{x}F)^2\;\mu(\mathrm{d}x)<\infty$. By $N$ we denote a standard Gaussian random variable. Then,	
		\begin{align*}
			\mathrm{d}_W(F,N)\leq \gamma_1+\gamma_2+\gamma_3\quad\text{ and }\quad \mathrm{d}_W(F,N)\leq \gamma_1+\gamma_2+\widetilde{\gamma}_3,
		\end{align*} 
		where
		\begin{align*}
			&\gamma_1=2\Big[\int(\E[(D_{x_1}F)^2(D_{x_2}F)^2])^{1/2}(\E[(D_{x_1,x_3}^2F)^2(D_{x_2,x_3}^2F)^2])^{1/2}\;\mu^3(\mathrm{d}(x_1,x_2,x_3))\Big]^{1/2},\\
			&\gamma_2=\Big[\int\E[(D_{x_1,x_3}^2F)^2(D_{x_2,x_3}^2F)^2]\;\mu^3(\mathrm{d}(x_1,x_2,x_3))\Big]^{1/2},\\
			&\gamma_3=\int\E[\lvert D_{x}F\rvert^3]\;\mu(\mathrm{d}x),\\
			&\widetilde{\gamma}_3=\int(\E[\lvert D_{x}F\rvert^3])^{1/3}(\E[\min\{\sqrt{8}\lvert D_xF\rvert^{3/2},\lvert D_xF\rvert^3\}])^{2/3}\;\mu(\mathrm{d}x).
		\end{align*}
	\end{theorem}
	The following similar result for the Kolmogorov distance can be found in \cite[Theorem 1.2]{LPS16}. 
	\begin{theorem}\label{thm:clt_kolmogorov_theory}
		Let $F\in L_\xi^2$ be a Poisson functional with $\E[F]=0$, $\Var[F]=1$ and $\E\int(D_{x}F)^2\;\mu(\mathrm{d}x)<\infty$. By $N$ we denote a standard Gaussian random variable. Then,	
		\begin{align*}
			\mathrm{d}_K(F,N)\leq \gamma_1+\gamma_2+\gamma_3+\gamma_4+\gamma_5+\gamma_6,
		\end{align*} 
		where $\gamma_1,\gamma_2,\gamma_3$ are defined as in Theorem \ref{thm:clt_theory_Wasserstein} and
		\begin{align*}
			&\gamma_4=\frac{1}{2}(\E[F^4])^{1/4}\int[(\E[(D_{x}F)^4])^{3/4}]\;\mu(\mathrm{d}x),\\
			&\gamma_5=\Big(\int\E[(D_{x}F)^4]\;\mu(\mathrm{d}x)\Big)^{1/2},\\
			&\gamma_6=\Big(\int6(\E[(D_{x_1}F)^4])^{1/2}(\E[(D_{x_1,x_2}^2F)^4])^{1/2}+3\E[(D_{x_1,x_2}^2F)^4]\;\mu^2(\mathrm{d}(x_1,x_2))\Big)^{1/2}.
		\end{align*}
	\end{theorem}
	
	\section{Main results}	
	This chapter provides the main results of this paper, which include the asymptotics of the expectation and the variance of geometric functionals of excursion sets and central limit theorems for growing observation windows $(W_r)_{r\geq1}$.
	\begin{theorem}
		\label{thm:intrinsicvolumes_expectation}
		Let $\varphi$ be a geometric functional and let $u>0$ be fixed. Under the first order moment condition \eqref{Assumption first moment}, we have
		\begin{align*}
			\lim\limits_{r\to\infty}\frac{\E[\varphi(Z_u\cap W_r)]}{V_d(W_r)}=\mathbb{E}[\varphi(Z_u\cap C_0^d)],
		\end{align*}
		where $C_0^d=[0,1)^d$ denotes the half-open $d$-dimensional unit cube and $\varphi(Z_u\cap C_0^d)$ is defined as in \eqref{eq:c0^d}.
	\end{theorem}
	\begin{rema}
		For $\varphi=V_d$ one can calculate the expectation directly. Since the Poisson process is stationary, $f_\eta(y)$ has the same distribution as $f_\eta(0)$. Therefore, $\p(y\in Z_u)=\p(0\in Z_u)$ for all $y\in W_r$ and we have
		\begin{align*}
			\E[V_d(Z_u\cap W_r)]=\E\int_{W_r}\mathbbm{1}\{y\in Z_u\}\;\mathrm{d}y=\p(0\in Z_u)V_d(W_r),
		\end{align*}
		which shows the result in Theorem \ref{thm:intrinsicvolumes_expectation} not only in the limit but for all $r\geq 1$ as $\p(0\in Z_u)$ does not depend on the observation window $W_r$.
	\end{rema}
	
	A similar result as in Theorem \ref{thm:intrinsicvolumes_expectation} was shown in \cite[Theorem 9.2.1]{SW08} for the limit of $\E[\varphi(Z\cap W_r)]$, where $Z$ is a so called standard random set. Remember that for a polyconvex set $A\in\mathcal{R}^d$, $N(A)$ denotes the smallest possible $M\in\mathbb{N}$ such that $A$ can be written as a union of $M$ compact convex sets, i.e.\ $A=\bigcup_{i=1}^MK_i$ for some compact convex sets $K_i$, $i\in\{1,\dots,M\}$. A standard random set $Z$ defined as in \cite[Definition 9.2.1]{SW08} fulfils the integrability condition $\E[2^{N(Z\cap C^d)}]<\infty$. 
	Our excursion sets do not necessarily fulfil this integrability condition, which is essential for the proof of the asymptotic expectation in \cite[Theorem 9.2.1]{SW08}, and are therefore not necessarily standard random sets. In fact, if $\eta(S_{C^d})=n$, the construction in Proposition \ref{lemma E polyconvex} may yield up to $2^{n}-1$ convex sets and if $N(Z_u\cap C^d)$ becomes with high probability too large, $\E[2^{N(Z_u\cap C^d)}]$ is no longer finite.
	Since Proposition \ref{lemma E polyconvex} only shows an upper bound for $N(Z_u\cap C^d)$, the following example shows that $\E[2^{N(Z_u\cap C^d)}]$ is indeed not necessarily finite.
	\begin{example}\label{example:notstandard}
		Let $\mathbb{Q}$ be the probability measure on $\mathbb{M}=\mathbb{N}$ with $\Q(\{k\})=(1-p)p^{k-1}$ for $k\in\mathbb{N}$ and some fixed $p\in(2^{-1/4},1)$, i.e. the marks are geometrically distributed with parameter $p$. We consider the case $\gamma=1$ and $d=2$, where for $m\in\mathbb{N}$ and $x\in\mathbb{R}^2$, $g_m:\mathbb{R}^2\to\mathbb{R}_{\geq 0}$ is  defined by
		\begin{align*}
			g_m(x)=\begin{cases}
				\frac{u}{2} &\text{ for }x\in K_m,\\
				0 &\text{ else,}
			\end{cases}
		\end{align*}
		where $K_{2j-1}=[-\frac{1}{4(2j-1)},\frac{1}{4(2j-1)}]\times[-1,1]$ and $K_{2j}=[-1,1]\times[-\frac{1}{4(2j-1)},\frac{1}{4(2j-1)}]$ for $j\in\mathbb{N}$. 
		One can show that $\E[2^{N(Z_u\cap C^2)}]=\infty$ by constructing configurations with $2i$ points and $N(Z_u\cap C^2)=i^2$ for all $i\in\mathbb{N}$, which arise with a sufficiently large probability. For a detailed construction see Section \ref{proof:example}.
	\end{example}
	Under a second order moment condition we now consider the asymptotic behaviour of the variance.
	\begin{theorem}\label{thm:variance_limit}
		Let $\varphi$ be a geometric functional and let $u>0$ be fixed. Assume \eqref{Assumption_fourth_moment} for $k=2$. Then, the limit
		\begin{align*}
			\lim\limits_{r\to\infty}\frac{\Var[\varphi(Z_u\cap W_r)]}{V_d(W_r)}=\sigma_0
		\end{align*}
		exists, is finite and given by
		\begin{align*}
			\sigma_0=\sum_{n=1}^\infty \frac{\gamma}{n!}\int_\MM \int_{(\mathbb{R}^d\times\MM)^{n-1}}(\E D^n_{(\mathbf{0},m),\hat{x}_2,\dots,\hat{x}_n}\varphi(Z_u\cap K))^2\;\lambda^{n-1}(\mathrm{d}(\hat{x}_2,\dots,\hat{x}_n))\;\mathbb{Q}(\mathrm{d}m),
		\end{align*}
		where $K=K_{m}\cap\bigcap_{j=2}^n\hat{K}_j$ and $\mathbf{0}=(0,\dots,0)\in\mathbb{R}^d$.
	\end{theorem} 
	The question remains under which conditions $\sigma_0>0$. In \cite[Section 4]{HLS16} this question was discussed for the special case of the Boolean model in a multivariate setting and for example in \cite{BST12,L19,LPY22,ST22} different conditions on the kernel functions were introduced to show the positivity of the asymptotic variance and derive central limit theorems for functionals like the volume of Poisson shot noise processes. With a similar argument as used in \cite[Section 5]{ST22} for the volume we can show for different families of kernel functions that the asymptotic variance is positive for specific geometric functionals.
	\begin{prop}\label{prop: variance positive}
		Let $\varphi$ be a geometric functional, $u>0$ and let $\sigma_0$ be defined as in Theorem \ref{thm:variance_limit}. Assume \eqref{Assumption_fourth_moment} for $k=2$.
		\begin{enumerate}
			\item [a)] If $\Q(\{m\in\MM:\varphi(\{x\in \mathbb{R}^d:g_m(x)\geq u\})\neq0\})>0$, it holds $\sigma_0>0.$ \label{prop_a}
			\item [b)] Assume that there exists $\alpha=(\alpha_0,\dots,\alpha_d)\in\mathbb{R}^d\backslash\{\mathbf{0}\}$ satisfying
			\begin{align}\label{Assumption:phi}
				\varphi=\sum_{i=0}^{d}\alpha_iV_i.
			\end{align} Then, if $\Q(\{m\in\MM:\max_{x\in\mathbb{R}^d}g_m(x)\leq u, g_m \text{ strictly concave on }K_m\})>0$, it holds $\sigma_0>0.$\label{prop_b}
		\end{enumerate}
	\end{prop}
	\begin{rema}
		Proposition \ref{prop: variance positive} mainly distinguishes the cases $\max_{x\in\mathbb{R}^d}g_m(x)\geq u$ and $\max_{x\in\mathbb{R}^d}g_m(x)\leq u$ for all $m\in M$ and some $M\subseteq\MM$ satisfying $\Q(M)>0$. In the first case, our proof works for general geometric functionals if $\varphi(\{x\in \mathbb{R}^d:g_m(x)\geq u\})\neq0$ for enough $m\in M$. In the second case we can show the positivity of the variance for geometric functionals that can be represented as a linear combination of intrinsic volumes if $g_m$ is strictly concave for enough $m\in M$. Note that by Hadwiger's theorem (see for example \cite[Theorem 14.4.6]{SW08}), this class of geometric functionals includes all geometric functionals that are continuous and invariant under rigid motions.
	\end{rema}
	Finally we derive a qualitative central limit theorem and corresponding quantitative central limit theorems in Wasserstein and Kolmogorov distance under different moment assumption if the asymptotic variance is positive.
	\begin{theorem}\label{thm:clt}
		Let $\varphi$ be a geometric functional and let $u>0$. Denote by $N$ a standard Gaussian random variable, let $\sigma_0$ be defined as in Theorem \ref{thm:variance_limit} and assume $\sigma_0>0$. 
		\begin{enumerate}
			\item[a)] If \eqref{Assumption_fourth_moment} is fulfilled for $k=2$, it holds
			\begin{align*}
				\frac{\varphi(Z_u\cap W_r)-\E[\varphi(Z_u\cap W_r)]}{\sqrt{\Var[\varphi(Z_u\cap W_r)]}}\overset{d}{\longrightarrow} N \quad\text{ as }\quad r\to\infty.
			\end{align*}
			\item[b)] If \eqref{Assumption_fourth_moment} is fulfilled for $k=3$, there exists a constant $C>0$, depending on $\gamma, d$ and the first three moments in \eqref{Assumption_fourth_moment}, such that  
			\begin{align*}
				d_W\Bigg(\frac{\varphi(Z_u\cap W_r)-\E[\varphi(Z_u\cap W_r)]}{\sqrt{\Var[\varphi(Z_u\cap W_r)]}},N\Bigg)\leq \frac{C}{\sqrt{V_d(W_r)}}
			\end{align*}
			for $r$ large enough.
			\item[c)] If \eqref{Assumption_fourth_moment} is fulfilled for $k=4$, there exists a constant $C>0$, depending on $\gamma, d$ and the first four moments in \eqref{Assumption_fourth_moment}, such that  
			\begin{align*}
				d_K\Bigg(\frac{\varphi(Z_u\cap W_r)-\E[\varphi(Z_u\cap W_r)]}{\sqrt{\Var[\varphi(Z_u\cap W_r)]}},N\Bigg)\leq \frac{C}{\sqrt{V_d(W_r)}}
			\end{align*}
			for $r$ large enough.
		\end{enumerate}
	\end{theorem}
	Similar results as in Theorem \ref{thm:variance_limit} and Theorem \ref{thm:clt} were already shown for geometric functionals of the Boolean model and specific geometric functionals of excursion sets of Poisson shot noise processes. The analogous results for the Boolean model and for Poisson cylinder processes can be found in \cite[Theorem 3.1 and Theorem 9.3]{HLS16} and in \cite[Theorem 3.5 and Theorem 3.10]{BST22}. For Poisson shot noise processes there exist various results for different functionals under several model assumptions. For example, for the volume or smoothed versions of the volume of excursion sets of Poisson shot noise processes central limit theorems were shown in \cite[Proposition 3.2.1]{BST12}, \cite[Theorem 4.1]{L19} and \cite[Theorem 4.3]{LPY22} under different conditions on the kernel functions.

	\section{Proofs}
	Similarly to the case of the Boolean model treated in \cite{HLS16} or the case of Poisson cylinder processes in \cite{BST22}, the following proofs are based on the Fock space representation and the Malliavin-Stein method. In particular, the proof for the variance asymptotics in Theorem \ref{thm:variance_limit} follows the strategy of the proof of \cite[Theorem 3.1]{HLS16} and the proof for the central limit theorems in Theorem \ref{thm:clt} uses similar arguments as the proof of \cite[Theorem 3.5]{BST22}. The main difference and difficulty compared to the proofs in \cite{BST22, HLS16} is that $\E[2^{N(Z_u\cap C^d)}]$ is in the case of Poisson shot noise processes not necessarily finite as explained in Section \ref{sec:boolean_model}. For this reason, for a set $K\in\mathcal{K}^d$ we use here a dynamic decomposition, which depends on the number of points whose corresponding supports hit but do not cover this set. Therefore, we define $\eta([Q]^*)$ for $Q\in\mathcal{K}^d$ as the number of points $\hat{x}\in\eta$ for which $K(\hat{x})\cap Q\notin\{\emptyset, Q\}$.
	
	For $n\in\mathbb{N}_0$ and $z\in\mathbb{Z}^d$ let $Q_{n,z}=2^{-n}(z+[0,1]^d)$ be the cube of side length $2^{-n}$ and left lower corner $2^{-n}z$. We start covering $K$ with $\bigcup_{z\in\mathbb{Z}^d:Q_{0,z}\cap K\neq \emptyset}Q_{0,z}$. Then, for $n=0,1,2,\dots$,  if $\eta([Q_{n,z}]^*)> L$ for a fixed constant $L\in\mathbb{N}$, we iteratively decompose $Q_{n,z}$ in $2^d$ cubes $Q_{n+1,p_1},\dots, Q_{n+1,p_{2^d}}$ with $\{p_1,\dots,p_{2^d}\}=\{ 2z+\sum_{j\in J}e_j\text{ for } J\subseteq\{1,\dots,d\}\}$, where $e_j$ denotes the unit vector in direction $j$.
	Thus, at the end we cover $K$ with a dynamic grid of cubes, i.e. we consider
	\begin{align*}
		\bigcup_{(n,z)\in I_{K,L}}Q_{n,z}
	\end{align*}
	with
	\begin{align*}
		I_{K,L}=\{(n,z)\in\mathbb{N}_0\times \mathbb{Z}^d: Q_{n,z}\cap K\neq \emptyset, \eta([Q_{n-1,\lfloor z/2\rfloor}]^*)> L, \eta([Q_{n,z}]^*)\leq L \},
	\end{align*}
	with $\eta([Q_{-1,\lfloor z/2\rfloor}]^*)=L+1$ and $\lfloor z/2\rfloor=(\lfloor z_1/2\rfloor,\dots, \lfloor z_d/2\rfloor)$ for $z=(z_1,\dots,z_d)\in\mathbb{Z}^d$, where $\lfloor \cdot\rfloor$ stands for the floor function. For $(m,p)\in I_{K,L}$ let 
	\begin{align*}
		N(m,p)=\{(n,z)\in I_{K,L}: Q_{m,p}\cap Q_{n,z}\neq \emptyset, n\leq m\}.
	\end{align*}
	Then, the number of adjacent cubes $\lvert N(m,p)\rvert$ which are not smaller than the current cube itself is bounded by $3^d$ for geometric reasons.
	
	\subsection{Expectation}
	We use the following lemma to control the probability that $(n,z)\in I_{K,L}$ for $n\in\mathbb{N}$ and $z\in\mathbb{Z}^d$.
	\begin{lemma}
		\label{lemma_prob_exp}
		For every $L\in\mathbb{N}$ there exists a constant $C>0$, only depending on the first moment in \eqref{Assumption first moment}, $d$, $L$ and $\gamma$, such that for all  $n\in\mathbb{N}$ and $z\in\mathbb{Z}^d$,
		\begin{align*}
			\p( \eta([Q_{n-1,\lfloor z/2\rfloor}]^*)> L, \eta([Q_{n,z}]^*)\leq L)\leq C2^{-n(L+1)}.
		\end{align*}
	\end{lemma}
	\begin{proof}
		For $n\in\mathbb{N}$ and $z\in\mathbb{Z}^d$ define 
		\begin{align*}
			P(n,z)&=\{\hat{x}\in \mathbb{R}^d\times \MM: Q_{n-1,z}\cap K(\hat{x})\notin\{\emptyset,Q_{n-1,z}\}\}.	
		\end{align*}
		Let $c_{n,z}$ denote the centre of $Q_{n-1,z}$, $\partial A$ the boundary of a set $A\subset\mathbb{R}^d$ and  $\mathrm{d}(A,y)$ the distance of a set $A\subseteq \mathbb{R}^d$ to a point $y\in\mathbb{R}^d$. Then, $Q_{n-1,z}\cap K(\hat{x})\notin\{\emptyset,Q_{n-1,z}\}$ implies that $\mathrm{d}(\partial K(\hat{x}),c_{n,z})\leq\frac{\sqrt{d}}{2^{n}}$, because $\partial K(\hat{x})$ has to intersect $Q_{n-1,z}$. Therefore, since $K(\hat{x})=K_{m}+x$ for $\hat{x}=(x,m)$,
		\begin{align*}
			\lambda(P(n,z))&
			\leq \gamma\int_\MM\lambda_d\Big(\Big\{x\in \mathbb{R}^d:\mathrm{d}(\partial K(\hat{x}),c_{n,z})\leq \frac{\sqrt{d}}{2^{n}}\Big\}\Big)\;\mathbb{Q}(\mathrm{d}m)\nonumber\\
			&= \gamma\int_\MM\lambda_d\Big(\Big\{x\in \mathbb{R}^d: \mathrm{d}(\partial K_{m},c_{n,z}-x)\leq \frac{\sqrt{d}}{2^{n}}\Big\}\Big)\;\mathbb{Q}(\mathrm{d}m)\\
			&= \gamma\int_\MM\lambda_d\Big(\Big\{x\in \mathbb{R}^d: \mathrm{d}(\partial K_{m},x)\leq \frac{\sqrt{d}}{2^{n}}\Big\}\Big)\;\mathbb{Q}(\mathrm{d}m),
		\end{align*}
		where the last step uses the translation invariance of the volume.
		For a fixed $m\in\MM$ it holds by \cite[Equation (3.19)]{HLS16} because of the convexity of $K_{m}$,
		\begin{align*}
			\lambda_d\Big(\Big\{x\in K_{m}:\mathrm{d}(\partial K_{m},x)\leq \frac{\sqrt{d}}{2^{n}}\Big\}\Big)&\leq \lambda_d\Big(\Big(K_{m}+B^d\Big(\mathbf{0},\frac{\sqrt{d}}{2^{n}}\Big)\Big)\Big\backslash K_{m}\Big)\\&= \lambda_d\Big(\Big\{x\in \mathbb{R}^d\backslash K_{m}:\mathrm{d}(\partial K_{m},x)\leq \frac{\sqrt{d}}{2^{n}}\Big\}\Big).
		\end{align*}
		Together with Steiner's formula \eqref{eq:steiner} we get
		\begin{align}
			\lambda(P(n,z))
			&\leq 2\gamma\int_\MM \lambda_d\Big(\Big(K_{m}+B^d\Big(\mathbf{0},\frac{\sqrt{d}}{2^{n}}\Big)\Big)\Big\backslash K_{m}\Big)\;\mathbb{Q}(\mathrm{d}m)\nonumber\\
			&=2\gamma\sum_{k=0}^{d-1}\Big(\frac{\sqrt{d}}{2^{n}}\Big)^{d-k}\kappa_{d-k}\int_\MM V_k(K_{m})\;\mathbb{Q}(\mathrm{d}m)\leq c2^{-n}\label{eq:lambda(A_1(n))}		
		\end{align}
		for a suitable constant $c>0$, which might depend on the first moment in \eqref{Assumption first moment}, $d$ and $\gamma$.
		This provides
		\begin{align*}
			\p(\eta([Q_{n-1,\lfloor z/2\rfloor}]^*)>L, \eta([Q_{n,z}]^*)\leq L)
			&\leq\p(\eta(P(n,z))> L)\\
			&= \sum_{ s\geq L+1}\p(\eta(P(n,z))=s)\nonumber\\
			&=\sum_{ s\geq L+1}\frac{\lambda(P(n,z))^{s}}{s!}e^{-\lambda(P(n,z))}\\
			&\leq \lambda(P(n,z))^{L+1}\sum_{ s\geq L+1}\frac{c^{s-L-1}}{s!},
		\end{align*}
		which shows together with \eqref{eq:lambda(A_1(n))} the lemma.
	\end{proof}
	
	In the following lemma we use our decomposition to bound moments of geometric functionals of excursion sets. Since we need an extended version of this lemma for the proof of the variance and the central limit theorem, we provide the following lemma in its generalised version right here. For this we define $Z_u(\xi)=\{y\in\mathbb{R}^d:f_\xi(y)\geq u\}$ for a set $\xi$ of points in $\mathbb{R}^d\times\MM$, where $f_\xi(y)=\sum_{(x,m)\in\xi}g_m(y-x)$.
	\begin{lemma}\label{lemma:V_i and I}
		Let $A=\{\hat{x}_1,\dots,\hat{x}_t\}$ be a set of $t$ points in $\mathbb{R}^d\times\MM$ for some $t\in\mathbb{N}_0$. Denote $\eta_A=\eta+\sum_{j=1}^{t}\delta_{\hat{x}_j}$. Then it holds for $K\in\mathcal{K}^d$ and $L\in\mathbb{N}$,
		\begin{align*}
			\lvert \varphi(Z_u(\eta_A)\cap K)\rvert\leq C 2^{2^{L+T}}\lvert I_{K,L}\rvert
		\end{align*}
		for some constant $C>0$, which may depend on $d$, and $T=\lvert \{j\in\{1,\dots, t\}:\hat{K}_j\cap K\neq\{\emptyset, K\}\}\rvert$. For $K_1,K_2\in\mathcal{K}^d$, $k_1,k_2\in\mathbb{N}_0$ and $L\in\mathbb{N}$ with $L\geq 2d\max\{k_1,k_2\}$ we have
		\begin{align*}
			\E[\lvert I_{K_1,L}\rvert^{k_1}\lvert I_{K_2,L}\rvert^{k_2}]\leq C_{k_1,k_2}V_d\Big(K_1^{\sqrt{d}} \Big)^{k_1}V_d\Big(K_2^{\sqrt{d}} \Big)^{k_2}
		\end{align*}
		for some constant $C_{k_1,k_2}>0$, which may depend on $k_1$, $k_2$, $d$, $\gamma$, $L$ and the first moment in \eqref{Assumption first moment} and where $K_i^{\sqrt{d}}=(K_i)^{\sqrt{d}}$ for $i\in\{1,2\}$.
	\end{lemma}
	\begin{proof}
		Denote by $\hat{x}_{t+1},\dots,\hat{x}_{t+\eta(S_{K})}$ the points of the Poisson process in $S_{K}$, where $S_{K}=\{\hat{x}\in\hat{\mathbb{R}}^d:K(\hat{x})\cap K\neq \emptyset\}$.
		Let $\emptyset\neq I\subseteq\{1,\dots, t+\eta(S_{K})\}$ and $X_I$ be defined as in \eqref{eq_XI} with the difference that $W_r$ is replaced by $K$. For a set $Q\in\mathcal{K}^d$ with $Q\subseteq K$ and $j\in\{1,\dots,t+\eta(S_{K})\}$ we have $X_{I}\cap Q\subseteq X_{I\cup\{j\}}\cap Q$ if $Q\subseteq \hat{K}_j$ and $X_{I\cup\{j\}}\cap Q=\emptyset$ if $Q\cap \hat{K}_j=\emptyset$. Thus, for $J_1=\{j\in\{1,\dots,t\}:\hat{K}_j\cap Q\neq \{Q,\emptyset\}\}$, $J_2=\{j\in\{t+1,\dots,t+\eta(S_{K})\}:\hat{K}_j\cap Q\neq \{Q,\emptyset\}\}$ and $M=\{j\in\{1,\dots,t+\eta(S_{K})\}:\hat{K}_j\cap Q=Q\}$ we have
		\begin{align*}
			Z_u(\eta_A)\cap Q =\bigcup_{ \substack{I\subseteq J_1\cup J_2:\\ I\cup M\neq\emptyset}}X_{I\cup M}\cap Q.
		\end{align*}
		Under the condition that $\eta([Q]^*)\leq L$ this is the union of at most $2^{L+\lvert J_1\rvert}$ sets. Note that for $Q\subseteq K$, it holds $\lvert J_1\rvert\leq T$. Hence, if $Q\subseteq K$ is contained in some translate of the unit cube we can bound $\lvert \varphi(Z_u(\eta_A)\cap Q)\rvert$ with \eqref{eq:locallybounded} and the inclusion exclusion principle \eqref{eq:intrinsic volume polyconvex set}  by $2^{2^{L+T}}M_\varphi$. Thus, we have for $K\in\mathcal{K}^d$,
		\begin{align*}
			\lvert \varphi(Z_u(\eta_A)\cap K)\rvert&= \Big\lvert   \varphi	\Big(Z_u(\eta_A)\cap K\cap\bigcup_{(n,z)\in I_{K,L}}Q_{n,z}\Big)\Big\rvert\\&\leq \sum_{\emptyset\neq J\subseteq I_{K,L}}\Big\lvert  \varphi\Big(Z_u(\eta_A)\cap K\cap\bigcap_{(n,z)\in J}Q_{n,z}\Big)\Big\rvert\\
			&\leq   \sum_{(m,p)\in I_{K,L}}\sum_{\substack{J\subseteq N(m,p):\\ (m,p)\in J}}\Big\lvert  \varphi\Big(Z_u(\eta_A)\cap K\cap\bigcap_{(n,z)\in J}Q_{n,z}\Big)\Big\rvert\\
			&\leq2^{3^d}2^{2^{L+T}}M_\varphi\lvert I_{K,L}\rvert,
		\end{align*}
		where the last inequality holds because of $\lvert N(m,p)\rvert\leq 3^d$ for $(m,p)\in I_{K,L}$. This provides the first part of the lemma. 
		
		Now, let $K_1,K_2\in\mathcal{K}^d$. For the second part of the lemma note at first that the statement is clear for $k_1=k_2=0$. For the other cases
		we define $\tilde{I}_s=\{(n,z)\in\mathbb{N}_0\times\mathbb{Z}^d: Q_{n,z}\cap K_s\neq\emptyset\}$ for $s=1,2$ and 
		\begin{align}\label{eq:ell_s}
			\ell_s=\lvert \{z\in\mathbb{Z}^d:Q_{0,z}\cap K_s\neq \emptyset\}\rvert\leq V_d(K_s+B^d(\mathbf{0},\sqrt{d}))= V_d\Big(K_s^{\sqrt{d}}\Big).
		\end{align}
		Then, for fixed $n\in\mathbb{N}$, there are up to $2^{nd}\lvert \{z\in\mathbb{Z}^d:Q_{0,z}\cap K_s\neq \emptyset\}\rvert=2^{nd}\ell_s$ vectors $z\in\mathbb{Z}^d$ such that $(n,z)\in\tilde{I}_s$.
		Together with Lemma \ref{lemma_prob_exp} we have for $s\in\{1,2\}$ and $k_s\geq 1$,
		\begin{align*}
			\E[\lvert I_{K_s,L}\rvert^{2k_s}]
			&= \sum_{\substack{(n_{i},z_{i})\in\tilde{I}_s,\\i\in\{1,\dots,2k_s\}}}\E[\mathbbm{1}\{(n_{i},z_{i})\in I_{K_s,L} \text{ for all } i\in\{1,\dots,2k_s\}]\\
			&\leq 2k_s\sum_{\substack{(n_{i},z_{i})\in\tilde{I}_s,\\ n_1\geq n_{i},\\i\in\{1,\dots,2k_s\}}}\E[\mathbbm{1}\{(n_1,z_1)\in I_{K_s,L} \}]\\
			&= 2k_s\sum_{\substack{(n_{i},z_{i})\in\tilde{I}_s,\\ n_1\geq n_{i},\\i\in\{1,\dots,2k_s\}}}\p(\eta([Q_{n_1-1,\lfloor z/2\rfloor}]^*)>L, \eta([Q_{n_1,z}]^*)\leq L)\\
			&\leq 2k_s\Bigg[\ell_s^{2k_s}+\sum_{n_1=1}^\infty\sum_{\substack{n_i\in\{0,\dots,n_1\},\\i\in\{2,\dots,2k_s\}}}\sum_{\substack{z_{i}\in\mathbb{Z}^d:(n_i,z_i)\in\tilde{I}_s,\\i\in\{1,\dots,2k_s\}}}C2^{-n_1(L+1)}\Bigg]\\
			&\leq 2k_s \ell_s^{2k_s}\left(1+C\sum_{n_1=1}^{\infty}(n_1+1)^{2k_s-1}2^{2k_sn_1d}2^{-n_1(L+1)}\right)=: C_{k_s}\ell_s^{2k_s}<\infty
		\end{align*}
		for $L\geq 2dk_s$ and some constant $C_{k_s}>0$. Then, the Cauchy inequality implies for $L\geq 2d\max\{k_1,k_2\}$,
		\begin{align*}
			\E[\lvert I_{K_1,L}\rvert^{k_1}\lvert I_{K_2,L}\rvert^{k_2}]\leq \E[\lvert I_{K_1,L}\rvert^{2k_1}]^{1/2}\E[\lvert I_{K_2,L}\rvert^{2k_2}]^{1/2}\leq \sqrt{C_{k_1}C_{k_2}}\ell_1^{k_1}\ell_2^{k_2},
		\end{align*}
		which provides together with \eqref{eq:ell_s} the lemma for $k_1,k_2\geq 1$. The cases $k_1=0,k_2\geq 1$ and $k_1\geq1,k_2=0$ follow analogously.
	\end{proof}
	\begin{proof}[Proof of Theorem \ref{thm:intrinsicvolumes_expectation}]
		By Lemma \ref{lemma:V_i and I} we have for $\mathcal{K}^d\ni K\subseteq C^d$,
		\begin{align}\label{eq:cond_bounded}
			\E[\lvert\varphi(Z_u\cap K)\rvert]\leq cV_d(K^{\sqrt{d}})\leq cV_d(C^d+B^d(\mathbf{0},\sqrt{d}))<\infty
		\end{align}
		for some constant $c>0$, i.e.\ $\varphi(Z_u\cap K)$ is integrable. By the translation invariance and the additivity of $\varphi$, $\varphi(Z_u\cap A)$ is also integrable for $A\in\mathcal{R}^d$. Hence, we can define a function $\phi:\mathcal{R}^d\to\mathbb{R}$ by
		\begin{align*}
			\phi(A)=\E[\varphi(Z_u\cap A)].
		\end{align*}
		This function is clearly additive, translation invariant by the stationarity of $Z_u$ and locally bounded due to \eqref{eq:cond_bounded}. Hence, we can apply \cite[Lemma 9.2.2]{SW08}, which shows Theorem \ref{thm:intrinsicvolumes_expectation}.
	\end{proof}
	\subsection{Construction of Example \ref{example:notstandard}}
	\label{proof:example}
	Let $\MM,\mathbb{Q}$ and $g_m$ for $m\in\MM$ be defined as in Example \ref{example:notstandard}.
	In the following we show that for $p\in(2^{-1/4},1)$,  $\E[2^{N(Z_u\cap C^2)}]$ is not finite. To this end we construct configurations with $\eta(S_{C^2})=2i$ for all $i\in\mathbb{N}$, which arise with a sufficiently large probability and for which $N(Z_u\cap C^2)=i^2$. The idea of the following construction is that we divide the $2i$ points in two groups of size $i$ such that the corresponding shifted supports of two points from different groups overlap in a cube, which does not intersect the shifted support of any other point and hence provides that $N(Z_u\cap C^2)=i^2$.
	
	To this end,  for $i\in\mathbb{N}$ and $k\in\{1,\dots,i\}$ let $R_{2i-1}^{(k)}=[\frac{8k-7}{4(2i-1)},\frac{8k-5}{4(2i-1)}]\times [0,1]\times\{2i-1\}$ and $R_{2i}^{(k)}=[0,1]\times[\frac{8k-7}{4(2i-1)},\frac{8k-5}{4(2i-1)}]\times\{2i\}$. Then, if for some fixed $i\in\mathbb{N}$,
	\begin{align*}
		\eta(R_{2i-1}^{(k)})=\eta(R_{2i}^{(k)})=1 \quad\text{ for }\quad k\in\{1,\dots,i\} \quad\text{ and }\quad	\eta\Big(S_{C^2}\Big\backslash \bigcup_{k=1}^i(R_{2i-1}^{(k)}\cup R_{2i}^{(k)})\Big)=0,
	\end{align*} 
	it holds $N(Z_u\cap C^2)=i^2$. Since
	\begin{align}
		\lambda(R_{2i-1}^{(k)})&=\lambda_2\otimes\Q(R_{2i-1}^{(k)})=\frac{2}{4(2i-1)}\Q(\{2i-1\})=\frac{1-p}{2(2i-1)}p^{2i-2}\leq \frac{(1-p)}{2}\label{eq:R2i-1},\\
		\lambda(R_{2i-1}^{(k)})&\geq \lambda(R_{2i}^{(k)})=\frac{1-p}{2(2i-1)}p^{2i-1}\label{eq:R2i}
	\end{align}
	and
	\begin{align}\label{eq:SCd}
		\lambda\Big(S_{C^2}\Big\backslash \bigcup_{k=1}^i(R_{2i-1}^{(k)}\cup R_{2i}^{(k)})\Big)\leq \lambda(S_{C^2})\leq \lambda_2([-2,2]^2)=16,
	\end{align}
	we have
	\begin{align*}
		&\E[2^{N(Z_u\cap C^2)}]\geq \sum_{i=1}^{\infty}2^{i^2}\p(N(Z_u\cap C^2)=i^2,\eta(S_{C^2})=2i)\\
		&\geq\sum_{i=1}^{\infty}2^{i^2}\p\Big(	\eta(R_{2i-1}^{(k)})=\eta(R_{2i}^{(k)})=1\text{ for } k\in\{1,\dots,i\}, 	\eta\Big(S_{C^2}\Big\backslash \bigcup_{k=1}^i(R_{2i-1}^{(k)}\cup R_{2i}^{(k)})\Big)=0\Big)\\
		&\geq\sum_{i=1}^{\infty}2^{i^2}\prod_{k=1}^{i}\lambda(R_{2i-1}^{(k)})\lambda(R_{2i}^{(k)})e^{-\lambda(R_{2i-1}^{(k)})-\lambda(R_{2i}^{(k)})}e^{-\lambda(S_{C^2}\backslash \cup_{k=1}^i(R_{2i-1}^{(k)}\cup R_{2i}^{(k)}))}\\
		&\geq c \sum_{i=1}^{\infty}2^{i^2}\prod_{k=1}^{i}\lambda(R_{2i}^{(k)})^2
	\end{align*}
	for a suitable constant $c>0$, where $c$ depends on the bounds in \eqref{eq:R2i-1} and \eqref{eq:SCd}. Hence, with \eqref{eq:R2i},
	\begin{align*}
		\E[2^{N(Z_u\cap C^2)}]&\geq c\sum_{i=1}^{\infty}2^{i^2}\frac{(1-p)^{2i}}{2^{2i}(2i-1)^{2i}}p^{2i(2i-1)}
		\geq c\sum_{i=1}^{\infty}(2p^4)^{i^2}\Big(\frac{(1-p)^2}{p^2\cdot 4(2i-1)^2}\Big)^i=\infty
	\end{align*}
	because $2p^4>1$, which shows that $Z_u$ is not a standard random set for $p\in(2^{-1/4},1)$.
	
	\subsection{Variance}  
	In order to calculate the asymptotic variance using the Fock space representation, we need to control the expectation of the $n$'th difference operator. For the proof of the central limit theorem we additionally require bounds for moments and products of difference operators, which we establish here as well.
	\begin{lemma}
		\label{lemma:diff_operator}
		Let $\hat{x}_j\in\mathbb{R}^d\times\mathbb{M}$ for $j\in\{1,\dots,n\}$ and $\hat{K}=\bigcap_{j=1}^n\hat{K}_j$. Then it holds
		\begin{align}
			D^n_{\hat{x}_1,\dots,\hat{x}_n} \varphi(Z_u\cap W_r)=D^n_{\hat{x}_1,\dots,\hat{x}_n} \varphi(Z_u\cap W_r\cap \hat{K})\label{eq:D^n}
		\end{align}
		and for $\hat{y}_\ell,\;\hat{z}_s\in\mathbb{R}^d\times\mathbb{M}$ for $\ell\in\{1,\dots,n_1\}$, $s\in\{1,\dots,n_2\}$, $n_1,n_2\in\mathbb{N}$ and  $k_1,k_2\in\mathbb{N}_0$,
		\begin{align}
			&\E[\lvert D^{n_1}_{\hat{y}_1,\dots,\hat{y}_{n_1}} \varphi(Z_u\cap W_r)\rvert^{k_1}\lvert D^{n_2}_{\hat{z}_1,\dots,\hat{z}_{n_2}} \varphi(Z_u\cap W_r)\rvert^{k_2}]\nonumber\\&\leq c_{k_1,k_2}2^{k_1n_1+k_2n_2}V_d\Big(\hat{A}_1^{\sqrt{d}}\cap W_r^{\sqrt{d}}\Big)^{k_1}V_d\Big(\hat{A}_2^{\sqrt{d}}\cap W_r^{\sqrt{d}}\Big)^{k_2}\label{eq:E[D^n]}
		\end{align}
		for some constant $c_{k_1,k_2}>0$, which depends on $k_1$ and $k_2$ and where $\hat{A}_1=\bigcap_{\ell=1}^{n_1}K(\hat{y}_\ell)$ and $\hat{A}_2=\bigcap_{s=1}^{n_2}K(\hat{z}_s)$.
	\end{lemma}
	\begin{proof}
		For $j\in\{1,\dots,n\}$ and $\hat{x}_j=(x_j,m_j)\in \mathbb{R}^d\times\MM$  we have $g_{m_j}\geq 0$ and $g_{m_j}(y-x_j)=0$ for $y\notin \hat{K}_j$. With the additivity of geometric functionals we get
		\begin{align*}
			&D_{\hat{x}_j} \varphi(Z_u\cap W_r)\\
			&= \varphi(\{y\in W_r:f_\eta(y)\geq u-g_{m_j}(y-x_j)\})- \varphi(\{y\in W_r:f_\eta(y)\geq u\})\\
			&= \varphi(\{y\in W_r:f_\eta(y)\geq u\}\cup \{y\in W_r\cap \hat{K}_j:f_\eta(y)\geq u-g_{m_j}(y-x_j)\})\\&\quad- \varphi(\{y\in W_r:f_\eta(y)\geq u\})\\
			&= \varphi(\{y\in W_r\cap \hat{K}_j:f_\eta(y)\geq u-g_{m_j}(y-x_j)\})- \varphi(\{y\in W_r\cap\hat{K}_j:f_\eta(y)\geq u\})\\
			&=D_{\hat{x}_j} \varphi(Z_u\cap W_r\cap \hat{K}_j).
		\end{align*}
		Applying this scheme iteratively shows Equation \eqref{eq:D^n}.
		
		Lemma \ref{lemma:V_i and I} and Equation \eqref{eq:D^n} provide for $L\in\mathbb{N}$ and $n\in\mathbb{N}$,
		\begin{align*}
			\lvert D^{n}_{\hat{x}_1,\dots,\hat{x}_{n}} \varphi(Z_u\cap W_r)\rvert&=	\lvert D^{n}_{\hat{x}_1,\dots,\hat{x}_{n}} \varphi(Z_u\cap W_r\cap \hat{K})\rvert\\&=\Big\lvert\sum_{J\subseteq\{1,\dots,n\}}(-1)^{n-\lvert J\rvert}\varphi\Big(Z_u\Big(\eta+\sum_{j\in J}\delta_{(x_j,m_j)}\Big)\cap W_r\cap\hat{K}\Big)\Big\rvert\\&\leq \sum_{k=0}^{n}\binom{n}{k}C 2^{2^{L}}\lvert I_{\hat{K}\cap W_r,L}\rvert=C 2^n2^{2^{L}}\lvert I_{\hat{K}\cap W_r,L}\rvert.
		\end{align*}
		Note that $T$ from Lemma \ref{lemma:V_i and I} vanishes since $\hat{K}_j\cap \hat{K}\cap W_r= \hat{K}\cap W_r$ for $j\in\{1,\dots,n\}$.
		This provides
		\begin{align*}
			&\lvert D^{n_1}_{\hat{y}_1,\dots,\hat{y}_{n_1}} \varphi(Z_u\cap W_r)\rvert^{k_1}\lvert D^{n_2}_{\hat{z}_1,\dots,\hat{z}_{n_2}} \varphi(Z_u\cap W_r)\rvert^{k_2}\\&\leq  \big(C 2^{2^{L}}\big)^{k_1+k_2}2^{k_1n_1+k_2n_2}\lvert I_{\hat{A}_1\cap W_r,L}\rvert^{k_1}\lvert I_{\hat{A}_2\cap W_r,L}\rvert^{k_2}.
		\end{align*}
		Now, Lemma \ref{lemma:V_i and I} completes the proof of \eqref{eq:E[D^n]} with a choice of $L=2d\max\{k_1,k_2\}$ since $V_d\Big((\hat{A}_i\cap W_r)^{\sqrt{d}}\Big)\leq V_d\Big(\hat{A}_i^{\sqrt{d}}\cap W_r^{\sqrt{d}}\Big)$ for $i\in\{1,2\}$.
	\end{proof}
	\begin{rema}\label{rema:difference operator}
		Note that for $k\in\mathbb{N}$, Lemma \ref{lemma:diff_operator} especially provides
		\begin{align*}
			\E[\lvert D_{\hat{x}_1,\dots,\hat{x}_n}^n\varphi(Z_u\cap \hat{K}\cap W_r)\rvert^k]\leq c_{k,0} 2^{kn} V_d(\hat{K}^{\sqrt{d}}\cap W_r^{\sqrt{d}})^k
		\end{align*}
		for $\hat{K}=\bigcap_{j=1}^n\hat{K}_j$ for all $r\geq1$, which implies that
		\begin{align*}
			\E[\lvert D_{\hat{x}_1,\dots,\hat{x}_n}^n\varphi(Z_u\cap \hat{K})\rvert^k]\leq c_{k,0} 2^{kn} V_d(\hat{K}^{\sqrt{d}})^k
		\end{align*}
		since $r$ can be chosen large enough such that $\hat{K}\subseteq W_r$.
	\end{rema}
	\begin{proof}[Proof of Theorem \ref{thm:variance_limit}]
		Since the second moment of $ \lvert \varphi(Z_u\cap W_r)\rvert$ exists by Lemma \ref{lemma:V_i and I},
		the Fock space representation \eqref{eq:fockspace} and Equation \eqref{eq:D^n} provide
		\begin{align*}
			\Var[ \varphi(Z_u\cap W_r)]&=\sum_{n=1}^\infty \frac{1}{n!}\int(\E D^n_{\hat{x}_1,\dots,\hat{x}_n} \varphi(Z_u\cap W_r))^2\;\lambda^n(\mathrm{d}(\hat{x}_1,\dots,\hat{x}_n))\\
			&=\sum_{n=1}^\infty \frac{1}{n!}\int(\E D^n_{\hat{x}_1,\dots,\hat{x}_n} \varphi(Z_u\cap W_r\cap\hat{K}))^2\;\lambda^n(\mathrm{d}(\hat{x}_1,\dots,\hat{x}_n)).
		\end{align*}
		Let
		\begin{align*}
			f_{n,r}(m_1)=\frac{\gamma}{V_d(W_r)}\int_{\mathbb{R}^d}\int_{(\mathbb{R}^d\times\MM)^{n-1}}(\E D^n_{\hat{x}_1,\dots, \hat{x}_n} \varphi(Z_u\cap W_r\cap\hat{K}))^2\;\lambda^{n-1}(\mathrm{d}(\hat{x}_2,\dots,\hat{x}_n))\;\mathrm{d}x_1
		\end{align*}
		and for $K=K_{m_1}\cap\bigcap_{j=2}^n\hat{K}_j$ and $\mathbf{0}=(0,\dots,0)\in\mathbb{R}^d$,
		\begin{align}\label{eq:fn_umgeschrieben}
			f_n(m_1) &= \gamma\int_{(\mathbb{R}^d\times\MM)^{n-1}}(\E D^n_{(\mathbf{0},m_1),\hat{x}_2,\dots,\hat{x}_n} \varphi(Z_u\cap K))^2\;\lambda^{n-1}(\mathrm{d}(\hat{x}_2,\dots,\hat{x}_n))\nonumber\\
			&=\frac{\gamma}{V_d(W_r)}\int_{\mathbb{R}^d}\mathbbm{1}\{x_1+z\in W_r\}\nonumber\\&\quad\quad\quad\quad\quad\times\int_{(\mathbb{R}^d\times\MM)^{n-1}}(\E D^n_{(\mathbf{0},m_1),\hat{x}_2,\dots,\hat{x}_n} \varphi(Z_u\cap K))^2\;\lambda^{n-1}(\mathrm{d}(\hat{x}_2,\dots,\hat{x}_n))\;\mathrm{d}x_1\nonumber\\
			&=\frac{\gamma}{V_d(W_r)}\int_{\mathbb{R}^d}\mathbbm{1}\{x_1+z\in W_r\}\nonumber\\&\quad\quad\quad\quad\quad\times\int_{(\mathbb{R}^d\times\MM)^{n-1}}(\E D^n_{(x_1,m_1),\hat{x}_2,\dots,\hat{x}_n} \varphi(Z_u\cap K))^2\;\lambda^{n-1}(\mathrm{d}(\hat{x}_2,\dots,\hat{x}_n))\;\mathrm{d}x_1
		\end{align}
		for any $z\in K_{m_1}$ by translation invariance.
		We want to show that
		\begin{align*}
			\frac{\Var[\varphi(Z_u\cap W_r)]}{V_d(W_r)}=\sum_{n=1}^\infty \frac{1}{n!}\int_\MM f_{n,r}(m_1)\;\mathbb{Q}(\mathrm{d}m_1)\to \sum_{n=1}^\infty \frac{1}{n!}\int_\MM f_{n}(m_1)\;\mathbb{Q}(\mathrm{d}m_1)=\sigma_0
		\end{align*}
		for $r\to\infty$. To this end, we first apply Fubini and then show the convergence of the integrals using the dominated convergence theorem.
		
		To apply the dominated convergence theorem we start with bounding $\lvert f_{n,r}(m_1)\rvert$. 
		With Remark \ref{rema:difference operator} we have together with the monotonicity of intrinsic volumes of convex sets and \eqref{eq:steiner},
		\begin{align*}
			\lvert f_{n,r}(m_1)\rvert&\leq\frac{\gamma}{V_d(W_r)}\int_{\mathbb{R}^d}\int_{(\mathbb{R}^d\times\MM)^{n-1}} V_d(\hat{K}^{\sqrt{d}}\cap W_r^{\sqrt{d}})^2c_{1,0}^24^n\;\lambda^{n-1}(\mathrm{d}(\hat{x}_2,\dots,\hat{x}_n))\;\mathrm{d}x_1\\
			&\leq \frac{C_1c_{1,0}^24^n\gamma}{V_d(W_r)}\int_{\mathbb{R}^d}V_d(\hat{K}_1^{\sqrt{d}}\cap W_r^{\sqrt{d}})\int_{(\mathbb{R}^d\times\MM)^{n-1}}\sum_{j=0}^{d}V_j(\hat{K})\;\lambda^{n-1}(\mathrm{d}(\hat{x}_2,\dots,\hat{x}_n))\;\mathrm{d}x_1.
		\end{align*}
		Equation \eqref{eq:int_Acapk1capkn} and the translation invariance of intrinsic volumes provide
		\begin{align*}
			\int_{(\mathbb{R}^d\times\MM)^{n-1}}\sum_{j=0}^{d}V_j\Big(\hat{K}_1\cap\bigcap_{j=2}^n\hat{K}_j\Big)\;\lambda^{n-1}(\mathrm{d}(\hat{x}_2,\dots,\hat{x}_n))\leq C_4^{n-1}\sum_{k=0}^{d}V_k(K_{m_1}).
		\end{align*}
		With \eqref{eq:steiner}, \eqref{eq:int_V_d A_1+xcapA_2}  and
		\begin{align*}
			V_d(W_r^{\sqrt{d}})=V_d(rW+B^d(\mathbf{0},\sqrt{d}))\leq r^d V_d(W+B^d(\mathbf{0},\sqrt{d}))=r^dV_d(W^{\sqrt{d}})
		\end{align*}
		for $r\geq1$, we have
		\begin{align*}
			\int_{\mathbb{R}^d}V_d(\hat{K}_1^{\sqrt{d}}\cap W_r^{\sqrt{d}})\;\mathrm{d}x_1&=\int_{\mathbb{R}^d}V_d((K_{m_1}^{\sqrt{d}}+x_1)\cap W_r^{\sqrt{d}})\;\mathrm{d}x_1=V_d(K_{m_1}^{\sqrt{d}}) V_d(W_r^{\sqrt{d}})\\&\leq C_1\sum_{k=0}^{d}V_k(K_{m_1})r^dV_d(W^{\sqrt{d}}).
		\end{align*}
		
		Altogether, this leads with the positivity of the intrinsic volumes for convex sets and $V_d(rW)=r^dV_d(W)$ to
		\begin{align}
			\sum_{n=1}^\infty\frac{1}{n!}\lvert f_{n,r}(m_1)\rvert&\leq\sum_{n=1}^\infty\frac{\gamma c_{1,0}^24^nC_4^{n-1}}{n!} C_1^2\frac{V_d(W^{\sqrt{d}})}{V_d(W)}\Big(\sum_{k=0}^{d}V_k(K_{m_1})\Big)^2\nonumber\\
			&\leq e^{4C_4}\gamma c_{1,0}^2\frac{C_1^2V_d(W^{\sqrt{d}})}{C_4V_d(W)}\Big(\sum_{k=0}^{d}V_k(K_{m_1})\Big)^2,\label{eq:dom_convergent}
		\end{align}
		which is independent of $r$ and by the second moment assumption in \eqref{Assumption_fourth_moment} integrable.
		
		In the next step we bound $\lvert f_{n,r}(m_1)-f_n(m_1)\rvert$. By \eqref{eq:fn_umgeschrieben} it holds for any $z\in K_{m_1}$,
		\begin{align*}
			&\lvert f_{n,r}(m_1)-f_n(m_1)\rvert\\&=\Big\lvert \frac{\gamma}{V_d(W_r)}\int_{\mathbb{R}^d}\int_{(\mathbb{R}^d\times\MM)^{n-1}}\Big[(\E D^n_{\hat{x}_1,\dots,\hat{x}_n} \varphi(Z_u\cap W_r\cap\hat{K}))^2\\&\hspace*{2.5cm}-\mathbbm{1}\{x_1+z\in W_r\}(\E D^n_{\hat{x}_1,\dots,\hat{x}_n} \varphi(Z_u\cap\hat{K}))^2\Big]\;\lambda^{n-1}(\mathrm{d}(\hat{x}_2,\dots,\hat{x}_n))\;\mathrm{d}x_1\Big\rvert.
		\end{align*}
		The term in the integral becomes $0$ if $\hat{K}_1\subseteq W_r$ or $\hat{K}_1\cap W_r=\emptyset$. In the first case it vanishes because for $\hat{K}_1\subseteq W_r$ it holds $ \varphi(Z_u\cap\hat{K})= \varphi(Z_u\cap W_r\cap\hat{K})$ and $x_1+z\in W_r$. In the second case it vanishes since for $\hat{K}_1\cap W_r=\emptyset$ we have $ \varphi(Z_u\cap W_r\cap\hat{K})=0$ and $x_1+z\notin W_r$. Together with Remark \ref{rema:difference operator} and the triangle inequality it holds
		\begin{align*}
			&\lvert f_{n,r}(m_1)-f_n(m_1)\rvert\\&=\Big\lvert \frac{\gamma}{V_d(W_r)}\int_{\mathbb{R}^d}\int_{(\mathbb{R}^d\times\MM)^{n-1}}\mathbbm{1}\{\hat{K}_1\cap\partial W_r\neq \emptyset\}\Big[(\E D^n_{\hat{x}_1,\dots,\hat{x}_n} \varphi(Z_u\cap W_r\cap\hat{K}))^2\\&\hspace*{2.5cm}-\mathbbm{1}\{x_1+z\in W_r\}(\E D^n_{\hat{x}_1,\dots,\hat{x}_n} \varphi(Z_u\cap\hat{K}))^2\Big]\;\lambda^{n-1}(\mathrm{d}(\hat{x}_2,\dots,\hat{x}_n))\;\mathrm{d}x_1\Big\rvert\\
			&\leq \frac{\gamma}{V_d(W_r)}\int_{\mathbb{R}^d}\mathbbm{1}\{\hat{K}_1\cap\partial W_r\neq \emptyset\}\\&\hspace*{2.5cm}\times\int_{(\mathbb{R}^d\times\MM)^{n-1}}2V_d(\hat{K}^{\sqrt{d}})^2c_{1,0}^24^n \;\lambda^{n-1}(\mathrm{d}(\hat{x}_2,\dots,\hat{x}_n))\;\mathrm{d}x_1.
		\end{align*}
		Together with the monotonicity of the intrinsic volumes and Equations \eqref{eq:steiner} and \eqref{eq:int_Acapk1capkn},  we have for the second part of the integral
		\begin{align*}
			&\int_{(\mathbb{R}^d\times\MM)^{n-1}}2V_d(\hat{K}^{\sqrt{d}})^2c_{1,0}^24^n \;\lambda^{n-1}(\mathrm{d}(\hat{x}_2,\dots,\hat{x}_n))\\
			&\leq 2c_{1,0}^24^nV_d(K_{m_1}^{\sqrt{d}})\int_{(\mathbb{R}^d\times\MM)^{n-1}} V_d(\hat{K}^{\sqrt{d}})\;\lambda^{n-1}(\mathrm{d}(\hat{x}_2,\dots,\hat{x}_n))\\
			&\leq 2C_4^{n-1}c_{1,0}^24^nV_d(K_{m_1}^{\sqrt{d}})C_1\sum_{j=0}^{d}V_j(K_{m_1}).
		\end{align*}
		Together with Equations \eqref{eq:lambda_A_1+xcappartialA_2} and \eqref{eq:frationV_iV_d} this leads to
		\begin{align*}
			\lvert f_{n,r}(m_1)-f_n(m_1)\rvert&\leq 2\gamma C_2C_4^{n-1}c_{1,0}^24^nC_1V_d(K_{m_1}^{\sqrt{d}})\Big(\sum_{j=0}^{d}V_j(K_{m_1})\Big)^2\sum_{k=0}^{d-1}\frac{V_k(W_r)}{V_d(W_r)}\\
			&\leq 2\gamma C_2C_3C_4^{n-1}c_{1,0}^24^nC_1V_d(K_{m_1}^{\sqrt{d}})\Big(\sum_{j=0}^{d}V_j(K_{m_1})\Big)^2\frac{d}{r(W_r)}
		\end{align*}
		for $r\geq 1$ large enough such that $r(W_r)\geq 1$.	Hence,
		\begin{align*}
			\Big\lvert\sum_{n=1}^\infty\frac{1}{n!}f_{n,r}(m_1)-\sum_{n=1}^\infty\frac{1}{n!}f_n(m_1)\Big\rvert\leq e^{4C_4}\gamma c_{1,0}^2\frac{2C_1C_2C_3}{C_4}V_d(K_{m_1}^{\sqrt{d}})\Big(\sum_{j=0}^{d}V_j(K_{m_1})\Big)^2\frac{d}{r(W_r)},
		\end{align*}
		which goes to $0$ for $r\to\infty$ and provides the theorem together with \eqref{eq:dom_convergent} by the dominated convergence theorem.
	\end{proof}
	For the proof of the positivity of the variance, we additionally use three geometric lemmas for which we introduce the following notation. Let $m_1,\dots,m_k\in\MM$ and $x_k\in\mathbb{R}^d$ for some $k\in\mathbb{N}$ with $k\geq 2$. For $\mathbf{v}=(v_1,\dots,v_{k-1})\in(\mathbb{R}^d)^{k-1}$ let $\mathbf{K}(\mathbf{v})=\bigcap_{i=1}^{k-1}(K_{m_i}+v_i)\cap (K_{m_k}+x_k)$ and let $f_{\mathbf{v}}:\mathbb{R}^d\to\mathbb{R}_{\geq 0}$ with $f_{\mathbf{v}}(y)=\sum_{i=1}^{k-1}g_{m_i}(y-v_i)+g_{m_k}(y-x_k)$. Moreover, we define
	\begin{align*}
		\widetilde{Z}_u(\mathbf{v})=\Big\{y\in \mathbf{K}(\mathbf{v}):f_{\mathbf{v}}(y)\geq u\Big\}.
	\end{align*}
	\begin{lemma}\label{lemma:geom1 untere varianz}
		Let $k\in\mathbb{N}$ with $k\geq 2$, $x_k\in\mathbb{R}^d$ and $m_1,\dots,m_k\in \MM$ be such that $g_{m_i}$ is strictly concave on $K_{m_i}$ and $\max_{x\in\mathbb{R}^d}g_{m_i}(x)\leq\frac{u}{k-1}$ for $i\in\{1,\dots,k\}$. 
		Let $(\mathbf{x}_n)_{n\in\mathbb{N}}$ be a sequence in $(\mathbb{R}^d)^{k-1}$ with $\mathbf{x}_n\to\mathbf{x}$ as $n\to\infty$ for some $\mathbf{x}\in(\mathbb{R}^d)^{k-1}$. Then, for $t\in\{1,\dots,d\}$,
		\begin{align*}
			\limsup_{n\to\infty} V_t(\widetilde{Z}_u(\mathbf{x}_n))\leq  V_t(\widetilde{Z}_u(\mathbf{x})).
		\end{align*}
	\end{lemma}
	\begin{proof}
		Let $\mathbf{x}=(x_1,\dots,x_{k-1})\in(\mathbb{R}^d)^{k-1}$ and $\mathbf{x}_n=(x_{n,1},\dots,x_{n,k-1})\in(\mathbb{R}^d)^{k-1}$ for $n\in\mathbb{N}$ with $\mathbf{x}_n\to\mathbf{x}$ as $n\to\infty$. We start with showing that for each compact set $C\subseteq\mathbb{R}^d$,
		\begin{align}\label{zu cap C}
			\widetilde{Z}_u(\mathbf{x})\cap C=\emptyset\implies \widetilde{Z}_u(\mathbf{x}_n)\cap C=\emptyset 
		\end{align}
		for $n$ large enough. To this end let $C\subseteq\mathbb{R}^d$ be a compact set with $\widetilde{Z}_u(\mathbf{x})\cap C=\emptyset$ and define $C_1=\{y\in C:f_{\mathbf{x}}(y)=u\}$. For each $y\in C_1$ there exists $i\in\{1,\dots,k\}$ with $y\notin K_{m_i}+x_i$ because $y\notin \widetilde{Z}_u(\mathbf{x})$ and hence also $B^d(y,2\delta (y))\cap (K_{m_i}+x_i)=\emptyset$ for some $\delta(y)>0$. Since $K_{m_i}+x_{n,i}\to K_{m_i}+x_i$ as $n\to\infty$, there exists $N_1(y)\in\mathbb{N}$ with $B^d(y,\delta (y))\cap (K_{m_i}+x_{n,i})=\emptyset$ and hence $B^d(y,\delta (y))\cap \widetilde{Z}_u(\mathbf{x}_n)=\emptyset$ for $n\geq N_1(y)$. Note that by the strict concavity of $g_{m_i}$ on $K_{m_i}$ for $i\in\{1,\dots,k\}$, the maximum of $g_{m_i}$ is only attained at one point for $i\in\{1,\dots,k\}$. Together with $\max_{x\in\mathbb{R}^d}g_{m_i}(x)\leq\frac{u}{k-1}$ it follows  $\lvert C_1\rvert\leq 1$ (for $k>2$) or $\lvert C_1\rvert\leq 2$ (for $k=2$). This is due to the fact that $f_{\mathbf{x}}(y)=u$ outside of $\mathbf{K}(\mathbf{x})$ is only possible if for $\mathbf{x}=(x_1,\dots,x_{k-1})$,  $g_{m_i}(y-x_i)=\max_{x\in\mathbb{R}^d}g_{m_i}(x)$ for $k-1$ indices $i\in\{1,\dots,k\}$ and $g_{m_j}(y-x)=0$ for the remaining index $j$.
		
		Now, because of $\lvert C_1\rvert\leq 2$ there exists $N_1\in\mathbb{N}$ such that $\bigcup_{y\in C_1}B^d(y,\delta(y))\cap\widetilde{Z}_u(\mathbf{x}_n)=\emptyset$ for $n\geq N_1$. Let $C_2=C\backslash \mathrm{int}\big(\bigcup_{y\in C_1}B^d(y,\delta(y))\big)$. Then, $f_{\mathbf{x}}(y)<u$ for all $y\in C_2$. Assume, there exists a sequence $(n_j)_{j\in\mathbb{N}}$ such that for all $j\in\mathbb{N}$ there exists $y_j\in C_2$ with $f_{\mathbf{x}_{n_j}}(y_j)\geq u$. Since $C_2$ is compact, there exists a subsequence $(y_{j_{\ell}})_{\ell\in\mathbb{N}}$ which converges to some $y\in C_2$ as $\ell\to\infty$. As $g_{m_i}$ is upper semi-continuous by the conditions on $g_{m_i}$ from Section \ref{section1.2} for $i\in\{1,\dots,k\}$, we have 
		\begin{align*}
			\limsup_{\ell\to\infty}f_{\mathbf{x}_{n_{j_{\ell}}}}(y_{j_{\ell}})\leq f_{\mathbf{x}}(y)<u,
		\end{align*}
		which is a contradiction to the assumption. Hence, $C_2\cap \widetilde{Z}_u(\mathbf{x}_n)=\emptyset$ for almost every $n\in\mathbb{N}$. Together with the result for $C_1$ there exists $N_0\in\mathbb{N}$ with $C\cap\widetilde{Z}_u(\mathbf{x}_n)=\emptyset$ for all $n\geq N_0$.
		
		Additionally, there exists a compact set $Z\subseteq \mathbb{R}^d$ such that $\widetilde{Z}_u(\mathbf{x}_n)\subseteq Z$ for all $n\in\mathbb{N}$ since $\mathbf{x}_n\to\mathbf{x}$ as $n\to\infty$.
		Then, if $\widetilde{Z}_u(\mathbf{x})=\emptyset$, applying \eqref{zu cap C} for $C=Z$ provides $\widetilde{Z}_u(\mathbf{x}_{n})=\emptyset$ for $n$ large enough. Thus, $V_t(\widetilde{Z}_u(\mathbf{x}_n))=V_t(\widetilde{Z}_u(\mathbf{x}))=0$ for $n$ large enough in the case of $\widetilde{Z}_u(\mathbf{x})=\emptyset$.
		If $\widetilde{Z}_u(\mathbf{x})\neq\emptyset$, for any $\varepsilon>0$ we can use \eqref{zu cap C} for $C=Z\backslash \mathrm{int}(\widetilde{Z}_u(\mathbf{x})+B^d(\mathbf{0},\varepsilon))$, which yields the existence of $N_0(\varepsilon)\in\mathbb{N}$ with $\widetilde{Z}_u(\mathbf{x}_{n})\subseteq\widetilde{Z}_u(\mathbf{x})+B^d(\mathbf{0},\varepsilon)$ for $n\geq N_0(\varepsilon)$. Hence, $V_t(\widetilde{Z}_u(\mathbf{x}_{n}))\leq V_t(\widetilde{Z}_u(\mathbf{x})+B^d(\mathbf{0},\varepsilon))$ for any $\varepsilon>0$ if $n$ is large enough. Since $K\mapsto V_t(K)$ is continuous for all non-empty compact convex sets by \cite[Remark 3.22]{HW20}, this provides $\limsup_{n\to\infty} V_t(\widetilde{Z}_u(\mathbf{x}_n))\leq  V_t(\widetilde{Z}_u(\mathbf{x}))$.
	\end{proof}

	\begin{lemma}\label{lemma:h continuous}
		Let $k\in\mathbb{N}$ with $k\geq 2$, $x_k\in\mathbb{R}^d$ and $m_1,\dots,m_k\in \MM$. Define $h:(\mathbb{R}^d)^{k-1}\to\mathbb{R}_{\geq 0}$ by $h(\mathbf{x})=\max_{y\in \mathbb{R}^d}f_{\mathbf{x}}(y)$ for $\mathbf{x}\in(\mathbb{R}^d)^{k-1}$. Then, $h$ is continuous for all $\mathbf{x}\in(\mathbb{R}^d)^{k-1}$ with $r(\mathbf{K}(\mathbf{x}))>0$, where $r(\cdot)$ denotes the inradius. 
	\end{lemma}
	\begin{proof}
		For $\mathbf{\widetilde{x}}\in(\mathbb{R}^d)^{k-1}$ with $r(\mathbf{K}(\mathbf{\widetilde{x}}))>0$ let $y_{\max}(\mathbf{\widetilde{x}})$ be a point where the maximum of $f_{\mathbf{\widetilde{x}}}$ is attained.
		For $n\in\mathbb{N}$ let $\mathbf{x}_n\in(\mathbb{R}^d)^{k-1}$ with $\mathbf{x}_n\to\mathbf{x}$ as $n\to\infty$ and let $(y_{\max}(\mathbf{x}_{n_k}))_{k\in\mathbb{N}}$ be a sequence which converges to some $\tilde{y}$. Then, since $g_{m_i}$ is upper semi-continuous for $i\in\{1,\dots,k\}$, we have
		\begin{align*}
			\limsup\limits_{k\to\infty}f_{\mathbf{x}_{n_k}}(y_{\max}(\mathbf{x}_{n_k}))\leq f_{\mathbf{x}}(\tilde{y})\leq f_{\mathbf{x}}(y_{\max}(\mathbf{x}))=h(\mathbf{x}).
		\end{align*}
		Since $(y_{\max}(\mathbf{x}_{n}))_{n\in\mathbb{N}}$ is bounded, every subsequence of $(y_{\max}(\mathbf{x}_{n}))_{n\in\mathbb{N}}$ has a convergent subsequence. Thus,
		\begin{align*}
			\limsup\limits_{n\to\infty}h(\mathbf{x}_n)=\limsup\limits_{n\to\infty}f_{\mathbf{x}_n}(y_{\max}(\mathbf{x}_n))\leq h(\mathbf{x}).
		\end{align*}
		
		For the reverse direction we use that for any $\tau>0$ there exists some $y_{\tau}\in\mathbb{R}^d$ with $y_{\tau}\notin \partial(K_{m_i}+x_i)$ for $i\in\{1,\dots,k\}$ and such that $f_{\mathbf{x}}(y_{\tau})\geq f_{\mathbf{x}}(y_{\max}(\mathbf{x}))-\tau=h(\mathbf{x})-\tau$ as $r(\mathbf{K}(\mathbf{x}))>0$ and since $g_{m_i}|_{K_{m_i}}$ is continuous and $K_{m_i}$ is convex for all $i\in\{1,\dots,k\}$. Due to the continuity of $g_{m_i}$ in $y_{\tau}-x_i$ for $i\in\{1,\dots,k\}$, we get
		\begin{align*}
			\liminf_{n\to\infty}h(\mathbf{x}_n)\geq \liminf_{n\to\infty}f_{\mathbf{x}_n}(y_{\tau})=f_{\mathbf{x}}(y_{\tau})\geq h(\mathbf{x})-\tau
		\end{align*}
		for any $\tau>0$ and hence $\liminf_{n\to\infty}h(\mathbf{x}_n)\geq h(\mathbf{x})$. Altogether, $\lim\limits_{n\to\infty}h(\mathbf{x}_n)=h(\mathbf{x})$ for any $\mathbf{x}\in(\mathbb{R}^d)^{k-1}$ with $r(\mathbf{K}(\mathbf{x}))>0$.
	\end{proof}
	
	\begin{lemma}\label{lemma:geometric lemma}
		Let $k\in\mathbb{N}$ with $k\geq 2$, $x_k\in\mathbb{R}^d$ and $m_1,\dots,m_k\in \MM$ be such that $g_{m_i}$ is strictly concave on $K_{m_i}$ and $\frac{u}{k}<\max_{x\in\mathbb{R}^d}g_{m_i}(x)\leq\frac{u}{k-1}$ for $i\in\{1,\dots,k\}$. Then, for $t\in\{1,\dots,d\}$,
		\begin{align*}
			\lambda_d^{k-1}(\{\mathbf{x}\in(\mathbb{R}^d)^{k-1}:V_t(\widetilde{Z}_u(\mathbf{x}))\in(0,R)\})>0
		\end{align*}
		for all $R>0$.
	\end{lemma}
	\begin{proof}
		As in Lemma \ref{lemma:h continuous}  let $h:(\mathbb{R}^d)^{k-1}\to\mathbb{R}_{\geq 0}$ be defined by $h(\mathbf{x})=\max_{y\in \mathbb{R}^d}f_{\mathbf{x}}(y)$ for $\mathbf{x}\in(\mathbb{R}^d)^{k-1}$ and let $y_{\max}(\mathbf{x}) $ be a point where the maximum of $f_{\mathbf{x}}$ is attained. Note at first that $r(\widetilde{Z}_u(\mathbf{x}))>0$ for all $\mathbf{x}\in(\mathbb{R}^d)^{k-1}$ with $h(\mathbf{x})>u$ and $r(\mathbf{K}(\mathbf{x}))>0$ since $f_{\mathbf{x}}|_{\mathbf{K}(\mathbf{x})}$ is continuous by the continuity of $g_{m_i}|_{K_{m_i}}$ for $i\in\{1,\dots,k\}$ and because of $y_{\max}(\mathbf{x})\in\mathbf{K}(\mathbf{x})$ in this case as $\max_{x\in\mathbb{R}^d}g_{m_i}(x)\leq\frac{u}{k-1}$ for $i\in\{1,\dots,k\}$.
		
		Moreover, the space of all $\mathbf{x}\in(\mathbb{R}^d)^{k-1}$ satisfying $r(\mathbf{K}(\mathbf{x}))>0$ is connected for the following reason. 
		Let $c_i$ denote the centre of a ball of maximal radius inscribed in $K_{m_i}$ for $i\in\{1,\dots,k\}$ and let $\mathbf{v}=(v_1,\dots,v_{k-1})\in(\mathbb{R}^d)^{k-1}$ be such that $v_i+c_i=x_k+c_k$ for $i\in\{1,\dots,k-1\}$. Note that by the choice of $c_i$ this especially implies $r(\mathbf{K}(\mathbf{v}))>0$. Now, for each $\mathbf{x}=(x_1,\dots,x_{k-1})\in(\mathbb{R}^d)^{k-1}$ with $r(\mathbf{K}(\mathbf{x}))>0$ there exists a path from $\mathbf{x}$ to $\mathbf{v}$ such that $r(\mathbf{K}(\mathbf{w}))>0$ for all $\mathbf{w}=(w_1,\dots,w_{k-1})\in(\mathbb{R}^d)^{k-1}$ on the path by the convexity of $K_{m_1},\dots,K_{m_k}$. For example, such a path could be constructed in the following way. In a first step, $(x_1,\dots,x_{k-1})$ is shifted simultaneously by some $w_1\in\mathbb{R}^d$ to $\mathbf{w}_1=(x_1+w_1,\dots,x_{k-1}+w_1)\in(\mathbb{R}^d)^{k-1}$ such that $r(\mathbf{K}(\mathbf{w}))>0$ for all $\mathbf{w}\in(\mathbb{R}^d)^{k-1}$ on the path and $x_k+c_k\in \mathrm{int}(\mathbf{K}(\mathbf{w}_1))$ at the end. Now, iteratively for each $i\in\{1,\dots,k-1\}$ we can take as path the direct line to $v_i$. By the convexity of $K_{m_1},\dots, K_{m_k}$ it is guaranteed that $x_k+c_k\in \mathrm{int}(\mathbf{K}(\mathbf{w}))$ and hence $r(\mathbf{K}(\mathbf{w}))>0$ for all $\mathbf{w}\in(\mathbb{R}^d)^{k-1}$ on the path throughout the whole process, which shows that the space of all $\mathbf{x}\in(\mathbb{R}^d)^{k-1}$ satisfying $r(\mathbf{K}(\mathbf{x}))>0$ is connected.

		In the following we consider two cases.
		If there exists $\mathbf{x}\in(\mathbb{R}^d)^{k-1}$ with $h(\mathbf{x})\leq u$ and $r(\mathbf{K}(\mathbf{x}))>0$,
		we use that there exists $\mathbf{\bar{x}}\in(\mathbb{R}^d)^{k-1}$ satisfying $h(\mathbf{\bar{x}})>u$ and $r(\mathbf{K}(\mathbf{\bar{x}}))>0$ because of $\max_{x\in\mathbb{R}^d}g_{m_i}(x)>\frac{u}{k}$ for $i\in\{1,\dots,k\}$. Then, by the continuity of $h$ from Lemma \ref{lemma:h continuous} and the fact that the space of all $\mathbf{x}\in(\mathbb{R}^d)^{k-1}$ satisfying $r(\mathbf{K}(\mathbf{x}))>0$ is connected, there exist $\mathbf{\bar{z}}\in\mathbb{R}^d$ and a sequence $(\mathbf{\bar{z}}_n)_{n\in\mathbb{N}}$ with $\mathbf{\bar{z}}_n\to \mathbf{\bar{z}}$ as $n\to\infty$, $h(\mathbf{\bar{z}}_n)>u$, $h(\mathbf{\bar{z}})=u$, $r(\mathbf{K}(\mathbf{\bar{z}}_n))>0$ for all $n\in\mathbb{N}$ and $r(\mathbf{K}(\mathbf{\bar{z}}))>0$.
		Additionally, since $g_{m_i}$ is strictly concave on $K_{m_i}$ for $i\in\{1,\dots,k\}$, $f_{\mathbf{v}}$ is strictly concave on $\mathbf{K}(\mathbf{v})$ for any $\mathbf{v}\in(\mathbb{R}^d)^{k-1}$. Thus, there is at most one point in $\widetilde{Z}_u(\mathbf{\bar{z}})$, hence $V_t(\widetilde{Z}_u(\mathbf{\bar{z}}))=0$. As discussed in the first paragraph, note that the condition $h(\mathbf{\bar{z}}_n)>u$ for the sequence $(\mathbf{\bar{z}}_n)_{n\in\mathbb{N}}$ guarantees that $r(\widetilde{Z}_u(\mathbf{\bar{z}}_n))>0$ since  $r(\mathbf{K}(\mathbf{\bar{z}}_n))>0$ and hence, $V_t(\widetilde{Z}_u(\mathbf{\bar{z}}_n))>0$. 
		By Lemma \ref{lemma:geom1 untere varianz} it holds $0\leq \limsup_{n\to\infty}V_t(\widetilde{Z}_u(\mathbf{\bar{z}}_n))\leq V_t(\widetilde{Z}_u(\mathbf{\bar{z}}))=0$. Hence, for any $R>0$ there exists $n_R\in\mathbb{N}$ such that $0<V_t(\widetilde{Z}_u(\mathbf{\bar{z}}_{n_R}))\leq \frac{R}{2}$. By Lemma \ref{lemma:geom1 untere varianz} and Lemma \ref{lemma:h continuous}, there exists $\delta>0$ such that $h(\mathbf{x})>u$, $r(\mathbf{K}(\mathbf{x}))>0$ and $V_t(\widetilde{Z}_u(\mathbf{x}))<R $ for all $\mathbf{x}\in(\mathbb{R}^d)^{k-1}$ with $\lVert \mathbf{x}-\mathbf{\bar{z}}_{n_R}\rVert<\delta$ and thus $V_t(\widetilde{Z}_u(\mathbf{x}))\in(0,R) $ for all $\mathbf{x}\in(\mathbb{R}^d)^{k-1}$ with $\lVert \mathbf{x}-\mathbf{\bar{z}}_{n_R}\rVert<\delta$.
		This provides
		\begin{align*}
			&\lambda_d^{k-1}(\{\mathbf{x}\in(\mathbb{R}^d)^{k-1}:V_t(\widetilde{Z}_u(\mathbf{x}))\in(0,R)\})
			\geq  \lambda_d^{k-1}(\{\mathbf{x}\in(\mathbb{R}^d)^{k-1}:\lVert \mathbf{x}-\mathbf{\bar{z}}_{n_R}\rVert<\delta\})>0
		\end{align*}
		for any $R>0$.
		
		If $h(\mathbf{x})>u$ for all $\mathbf{x}\in(\mathbb{R}^d)^{k-1}$ with $r(\mathbf{K}(\mathbf{x}))>0$, it holds $V_t(\widetilde{Z}_u(\mathbf{x}))>0$ for any $\mathbf{x}\in(\mathbb{R}^d)^{k-1}$ with $r(\mathbf{K}(\mathbf{x}))>0$ by the arguments from the first paragraph.
		Since $\hat{K}_k$ is a compact convex set, for any $x\in\mathbb{R}^d\backslash\hat{K}_k$ there exists $y_k\in\hat{K}_k$ with $\lVert y_k-x\rVert=\max_{y\in \hat{K}_k}\lVert y-x\rVert$. Hence,
		the hyperplane $H$ through $y_k$, which is orthogonal to $y_k-x$ divides $\mathbb{R}^d$ in two half spaces $H^+$ and $H^-$, where $H^+$ denotes the half space that fulfils $\hat{K}_k\subseteq H^+$, and we have $H\cap \hat{K}_k=\{y_k\}$. 
		By choosing $\mathbf{\bar{z}}=(\bar{z}_1,\dots,\bar{z}_{k-1})\in(\mathbb{R}^d)^{k-1}$ such that $y_k\in K_{m_i}+\bar{z}_i$ and $K_{m_i}+\bar{z}_i\subseteq H^-$ for $i\in\{1,\dots,k-1\}$, we obtain $\mathbf{K}(\mathbf{\bar{z}})=\{y_k\}$.
		Clearly, there exists a sequence $(\mathbf{\bar{z}}_n)_{n\in\mathbb{N}}$ with $\mathbf{\bar{z}}_n\to\mathbf{\bar{z}}$ as $n\to\infty$ and $r(\mathbf{K}(\mathbf{\bar{z}}_n))>0$ for all $n\in\mathbb{N}$. Similarly to the first case, by Lemma \ref{lemma:geom1 untere varianz},  it holds $0\leq \limsup_{n\to\infty}V_t(\widetilde{Z}_u(\mathbf{\bar{z}}_n))\leq V_t(\widetilde{Z}_u(\mathbf{\bar{z}}))=0$, i.e.\ there exist $n_R\in\mathbb{N}$ and $\delta>0$ such that $V_t(\widetilde{Z}_u(\mathbf{x}))\in(0,R) $ for all $\mathbf{x}\in(\mathbb{R}^d)^{k-1}$ with $\lVert \mathbf{x}-\mathbf{\bar{z}}_{n_R}\rVert<\delta$. With 
		\begin{align*}
			&\lambda_d^{k-1}(\{\mathbf{x}\in(\mathbb{R}^d)^{k-1}:V_t(\widetilde{Z}_u(\mathbf{x}))\in(0,R)\})
			\geq  \lambda_d^{k-1}(\{\mathbf{x}\in(\mathbb{R}^d)^{k-1}:\lVert \mathbf{x}-\mathbf{\bar{z}}_{n_R}\rVert<\delta\})>0
		\end{align*}
		the proof of Lemma \ref{lemma:geometric lemma} is complete.
	\end{proof}

	\begin{proof}[Proof of Proposition \ref{prop: variance positive}] 
		We use Theorem \ref{thm:lowervarbound} for the difference operator of order $k$ for a suitable $k\in\mathbb{N}$. To verify the assumption in \eqref{Assumption_variance_theoretisch} we derive at first an upper bound for the integral of the expected squared difference operator of order $k+1$. To this end we use \eqref{eq:E[D^n]} and \eqref{eq:int_V_d A_1+xcapA_2}  to show that
		\begin{align*}
			&\E\idotsint \lvert D_{\hat{x}_1,\dots,\hat{x}_{k+1}}^{k+1}\varphi(Z_u\cap W_r)\rvert^2\;\lambda^{k+1}(\mathrm{d}(\hat{x}_1,\dots,\hat{x}_{k+1}))\\ &\leq \idotsint c_{2,0}4^{k+1}V_d\Big(\bigcap_{i=1}^{k+1}\hat{K}_i^{\sqrt{d}}\cap W_r^{\sqrt{d}}\Big)^2\;\lambda^{k+1}(\mathrm{d}(\hat{x}_1,\dots,\hat{x}_{k+1}))\\
			&\leq \gamma^{k+1}c_{2,0}4^{k+1}\int_{\MM^{k+1}}V_d\Big(K_{m_1}^{\sqrt{d}}\Big)\\&\hspace*{2.5cm}\times\int_{(\mathbb{R}^d)^ {k+1}}V_d\Big(\bigcap_{i=1}^{k+1}\hat{K}_i^{\sqrt{d}}\cap W_r^{\sqrt{d}}\Big)\;\mathrm{d}(x_1,\dots,x_{k+1})\;\mathbb{Q}^{k+1}(\mathrm{d}(m_1,\dots,m_{k+1}))\\
			&= \gamma^{k+1}c_{2,0}4^{k+1}\lambda_d\Big(W_r^{\sqrt{d}}\Big)\int_{\MM^{k+1}}V_d\Big(K_{m_1}^{\sqrt{d}}\Big)^2\prod_{i=2}^{k+1}V_d\Big(K_{m_i}^{\sqrt{d}}\Big)\;\mathbb{Q}^{k+1}(\mathrm{d}(m_1,\dots,m_{k+1}))\\
			&\leq c_1\lambda_d(W_r)
		\end{align*}
		for a suitable constant $c_1>0$. Note that due to the moment assumptions of up to order two, the integrals are finite. For the proof of the positivity of the variance it remains to show that 
		\begin{align*}
			&\E\idotsint \lvert D_{\hat{x}_1,\dots,\hat{x}_{k}}^{k}\varphi(Z_u\cap W_r)\rvert^2\;\lambda^{k}(\mathrm{d}(\hat{x}_1,\dots,\hat{x}_{k}))\geq c_2\lambda_d(W_r)
		\end{align*}
		for some $c_2>0$ and a suitable $k\in\mathbb{N}$. Then, assumption \eqref{Assumption_variance_theoretisch} is fulfilled and $\sigma_0>0$ by Theorem \ref{thm:lowervarbound}.
		
		Remember that $S_{K}=\{\hat{x}\in\hat{\mathbb{R}}^d:K(\hat{x})\cap K\neq \emptyset\}$ for $K\in\mathcal{K}^d$.
		For the proof of a) we choose $k=1$. Let $U_m=\{x\in\mathbb{R}^d:g_m(x)\geq u\}$ for $m\in \MM$ and $M=\{m\in \MM:\varphi(U_m)\neq 0\}$. By assumption, $\Q(M)>0$.
		Then, if $\eta(S_{K_m+x})=0$ for $m\in M$ and $x\in W_r$ such that $K_m+x\subseteq W_r$ we have with \eqref{eq:D^n},
		\begin{align*}
			D_{\hat{x}}\varphi(Z_u\cap W_r)=D_{\hat{x}}\varphi(Z_u\cap (K_m+x))=\varphi(U_m)-\varphi(\emptyset)=\varphi(U_m)
		\end{align*}
		for $\hat{x}=(x,m)$. This provides
		\begin{align}\label{eq:prop_lower_bound k=1}
			\p(	D_{\hat{x}}\varphi(Z_u\cap W_r)=\varphi(U_m))\geq \p(\eta(S_{K_m+x})=0)=e^{-\lambda(S_{K_m+x})}=e^{-\lambda(S_{K_m})}.
		\end{align}
		Let $r_0>0$ be such that $\Q(\{m\in M: R(K_m)\leq r_0\})>0$, where $R(K_m)$ denotes the circumradius of $K_m$, and $\widehat{W}_r=\{x\in W_r: \mathrm{dist}(x, \partial W_r)\geq r_0\}$. Then, $\lambda_d(\widehat{W}_r)\geq \frac{1}{2}\lambda_d(W_r)$ for $r$ large enough, which implies together with \eqref{eq:prop_lower_bound k=1} for $r$ large enough,
		\begin{align*}
			\E\int \lvert D_{\hat{x}}\varphi(Z_u\cap W_r)\rvert^2\;\lambda(\mathrm{d}\hat{x})
			&\geq \int\p( D_{\hat{x}}\varphi(Z_u\cap W_r)=  \varphi(U_m))\varphi(U_m)^2\;\lambda(\mathrm{d}\hat{x})\\
			&\geq\gamma \lambda_d(\widehat{W}_r)\int_M\mathbbm{1}\{R(K_m)\leq r_0\}e^{-\lambda(S_{K_m})}\varphi(U_m)^2\;\Q(\mathrm{d}m)\\
			&\geq c_3\lambda_d(W_r)
		\end{align*}
		for a suitable constant $c_3>0$, which completes the proof of a).
		
		For the proof of b) we use a similar proof strategy. This means, we start with point configurations, which occur with a positive probability and for which the excursion set on a specific compact convex set is the empty set. Then, we add just enough points such that it is not empty anymore and such that the corresponding difference operator is not equal to zero. 
		For the start let 
		\begin{align*}
			\widetilde{\MM}=\Big\{m\in\MM:\max_{x\in\mathbb{R}^d}g_m(x)\leq u, g_m \text{ strictly concave on }K_m\Big\}.
		\end{align*}
		Note that $\mathbb{Q}(\widetilde{\MM})>0$ by assumption. Moreover, let
		\begin{align*}
			M(v,\varepsilon)=\Big\{m\in\widetilde{\MM}:\max_{x\in\mathbb{R}^d}g_m(x)\in[v-\varepsilon,v+\varepsilon]\Big\}
		\end{align*}  
		for $v\in[0,u]$ and $\varepsilon>0$. Define the mapping $f:\widetilde{\MM}\to[0,u]$ given by $m\mapsto\max_{x\in\mathbb{R}^d}g_m(x)$ and denote by $\Q_f$ the corresponding push-forward measure. If $\Q_f$ has an atom in a point $v\in[0,u]$ with $\frac{u}{v}\in\mathbb{N}$ it holds $\Q(M(v,0))>0$. If $\Q_f$ does not have such a point, we can apply \cite[Lemma 1.19]{K21}, which provides that 
		\begin{align*}
			&\Q_f\Big(\Big\{v\in(0,u)\Big\backslash\bigcup_{i\in\mathbb{N}}\Big\{\frac{u}{i}\Big\}:\Q(M(v,\varepsilon))>0 \text{ for all }\varepsilon>0\Big\}\Big)\\&=	\Q_f(\{v\in[0,u]:\Q(M(v,\varepsilon))>0 \text{ for all }\varepsilon>0\})>0,
		\end{align*}
		where we used in the first step that $\Q_f(\{0\})=0$ since $K_m$ is non-empty for all $m\in\widetilde{\MM}$.
		Altogether, this guarantees that there either exists $v\in(0,u)$ with $\frac{u}{v}\notin\mathbb{N}$ such that $\Q(M(v,\varepsilon))>0$ for all $\varepsilon>0$ or there exists $v\in(0,u]$ with $\frac{u}{v}\in\mathbb{N}$ and $\Q(M(v,0))>0$.
		This will be used to divide the proof in different cases. 
		For the first both cases we further define $k\in\mathbb{N}$, $M\subseteq \widetilde{\MM}^{k}$ with $\Q^{k}(M)>0$ and 
		$S_{m_1,\dots,m_{k}}\subset(\mathbb{R}^d)^{k}$ for $(m_1,\dots,m_{k})\in M$ with 
		\begin{align}
			&(x_1,\dots,x_{k})\in S_{m_1,\dots,m_{k}} 	\Rightarrow (x_1+z,\dots,x_{k}+z)\in S_{m_1,\dots,m_{k}}\label{eq:prop1_Sm}
			\intertext{ for any $z\in\mathbb{R}^d$ and $x_1,\dots,x_{k}\in\mathbb{R}^d$,}
			&\lambda_d^{k-1}(\{(x_1,\dots,x_{k-1})\in(\mathbb{R}^d)^{k-1}:(x_1,\dots,x_{k-1},\mathbf{0})\in S_{m_1,\dots,m_{k}} \})>0\label{eq:prop2_Sm}
			\intertext{ in case that $k\geq 2$ and }
			&(x_1,\dots,x_{k})\notin S_{m_1,\dots,m_{k}} \text{ if }\bigcap_{i=1}^{k}\hat{K}_i=\emptyset.	\label{eq:prop3_Sm}
		\end{align}
		Let $t_{\min}=\min\{i\in\{0,\dots,d\}:\alpha_i\neq 0\}$. We consider the following three cases.
		
		\textbf{Case 1:}  There is some $v\in(0,u)$ with $\Q(M(v,\varepsilon))>0$ for any $\varepsilon>0$ and $\frac{u}{v}\notin\mathbb{N}$. 
		Then, we can choose $\varepsilon>0$ small enough such that there exists a $k\in\mathbb{N}$ with $k\geq 2$ satisfying 
		\begin{align*}
			\frac{u}{v-\varepsilon}< k<\frac{u}{v+\varepsilon}+1.
		\end{align*} 
		Now, we define $M=M(v,\varepsilon)^{k}$ and for $\hat{t}=\max\{t_{\min},1\}$,
		\begin{align*}
			S_{m_1,\dots,m_{k}}=\{(x_1,\dots,x_k)\in(\mathbb{R}^d)^{k}:V_{\hat{t}}(\widetilde{Z}_u(\mathbf{x}))\in (0,R)\text{ with } \mathbf{x}=(x_1,\dots,x_{k-1})\}
		\end{align*}
		for some $R>0$. Then, \eqref{eq:prop1_Sm}, \eqref{eq:prop2_Sm} and \eqref{eq:prop3_Sm} are fulfilled by the translation invariance of intrinsic volumes, Lemma \ref{lemma:geometric lemma} and since $V_{\hat{t}}(\emptyset)=0$.
		
		\textbf{Case 2:} We can find a $v\in(0,u]$ with $\frac{u}{v}\in\mathbb{N}$, $\Q(M(v,0))>0$ and $t_{\min}\neq 0$ or $v\neq u$. 
		Then, we define $k=\frac{u}{v}+1$, $M=M(v,0)^k$ and for $\hat{t}=\max\{t_{\min},1\}$ and  $U_m=\{x\in\mathbb{R}^d:g_m(x)\geq v\}$,
		\begin{align*}
			S_{m_1,\dots,m_{k}}=\Big\{&(x_1,\dots,x_{k})\in(\mathbb{R}^d)^{k}:V_{\hat{t}}(\widetilde{Z}_u(\mathbf{x}))\in (0,R) \text{ with }\mathbf{x}=(x_1,\dots,x_{k-1}),\\ &V_{t}\Big(\bigcap_{j=1}^{u/v} (U_{m_{i_j}}+x_{i_j})\Big)=0\text{ for } \{i_1,\dots, i_{u/v}\}\subseteq \{1,\dots,k\},\\&i_j\neq i_\ell \text{ for }j\neq \ell\text{ and all } t\geq t_{\min}\Big\},
		\end{align*}
		for some $R>0$.
		Note that since $U_m$ consists by the strict concavity of $g_m$ only of one point for all $m\in M(v,0)$, it holds 
		\begin{align*}
			\lambda_d^{u/v}\Big(\Big\{(x_{i_1},\dots,x_{i_{u/v}})\in(\mathbb{R}^d)^{u/v}: V_{t}\Big(\bigcap_{j=1}^{u/v} (U_{m_{i_j}}+x_{i_j})\Big)\neq0 \Big\}\Big)=0
		\end{align*}
		for $\{i_1,\dots, i_{u/v}\}\subseteq \{1,\dots,k\}$ and $i_j\neq i_\ell$ for $j\neq \ell$ and $t\geq t_{\min}$. Together with the translation invariance of the intrinsic volumes, properties \eqref{eq:prop1_Sm}, \eqref{eq:prop2_Sm} and \eqref{eq:prop3_Sm} are fulfilled as in Case $1$.
		
		\textbf{Case 3:} It holds $\Q(M(v,0))>0$ for $v=u$ and $t_{\min}=0$. Then, with \eqref{Assumption:phi}, we have for  $U_m=\{x\in\mathbb{R}^d:g_m(x)\geq v\}$,
		\begin{align*}
			\varphi(U_m)=\sum_{i=0}^d\alpha_iV_i(U_m)=\alpha_0\neq 0
		\end{align*}  for all $m\in M(v,0)$ since $U_m$ consists only of one point by the strict concavity of $g_m$ for all $m\in M(v,0)$. Hence, $\Q(\{m\in\MM:\varphi(U_m)\neq 0\})>0$. Thus, the assumption of part a) of the proposition is fulfilled and therefore the proof for Case $3$ is complete with the proof of a).
		
	 	It remains to show the proposition for Cases $1$ and $2$.
		To this end let $K=\bigcap_{i=1}^{k}\hat{K}_i\cap W_r$.
		By the choice of $k$ we have for $\eta(S_{K})=0$, $(x_1,\dots,x_{k})\in S_{m_1,\dots,m_{k}}$, $(m_1,\dots,m_{k})\in M$ and $I\subseteq\{1,\dots,k\}$,
		\begin{align*}
			Z_u\Big(\eta+\sum_{i\in I}\delta_{\hat{x}_i}\Big)\cap K=\emptyset
		\end{align*}
		if $\lvert I\rvert\leq k-1$ in Case $1$ and if $\lvert I\rvert\leq k-2$ in Case $2$. 
		If $\lvert I\rvert$ exceeds those thresholds, we have  
		\begin{align*}
			Z_u\Big(\eta+\sum_{i\in I}\delta_{\hat{x}_i}\Big)\cap K=X_{I}\cap K,
		\end{align*}
		which is convex by the proof of Proposition \ref{lemma E polyconvex}.
		
		Now, we consider $\hat{x}_1,\dots,\hat{x}_{k}$ with $\hat{K}_{i}\subseteq W_r$ for $i\in\{1,\dots,k\}$. Note that for $\eta(S_K)=0$ and fixed $x_k\in\mathbb{R}^d$ we have $Z_u\Big(\eta+\sum_{i=1}^k\delta_{\hat{x}_i}\Big)\cap K=\widetilde{Z}_u(\mathbf{x})$ for $\mathbf{x}=(x_1,\dots,x_{k-1})\in(\mathbb{R}^d)^{k-1}$.
 
		In the following we use assumption \eqref{Assumption:phi} to represent the geometric functional as a linear combination of intrinsic volumes. Recall that $t_{\min}=\min\{i\in\{0,\dots,d\}:\alpha_i\neq 0\}$. Without loss of generality we assume $\alpha_{t_{\min}}>0$. Then, for $t_{\min}>0$ it holds with \eqref{eq:ViVj},
		\begin{align*}
			\lvert\varphi(\widetilde{Z}_u(\mathbf{x}))\rvert&=\Big\lvert\sum_{i=0}^{d}\alpha_iV_i(\widetilde{Z}_u(\mathbf{x})\Big)\Big\rvert\\&\geq \sum_{i\in\{0,\dots,d\}:\alpha_i>0}\alpha_iV_i(\widetilde{Z}_u(\mathbf{x}))-\sum_{i\in\{0,\dots,d\}:\alpha_i<0}\lvert\alpha_i\rvert V_i(\widetilde{Z}_u(\mathbf{x}))\\
			&\geq V_{t_{\min}}(\widetilde{Z}_u(\mathbf{x})) \Big(\alpha_{t_{\min}}-\sum_{i\in\{0,\dots,d\}:\alpha_i<0}\lvert\alpha_i\rvert\frac{V_i(\widetilde{Z}_u(\mathbf{x}))}{V_{t_{\min}}(\widetilde{Z}_u(\mathbf{x}))}\Big)\\
			&\geq V_{t_{\min}}(\widetilde{Z}_u(\mathbf{x})) \Big(\alpha_{t_{\min}}-\sum_{i\in\{0,\dots,d\}:\alpha_i<0}\lvert\alpha_i\rvert C(t_{\min},i)V_{t_{\min}}(\widetilde{Z}_u(\mathbf{x}))^{(i-t_{\min})/t_{\min}}\Big)\\
			&\geq \frac{\alpha_{t_{\min}}}{2} V_{t_{\min}}(\widetilde{Z}_u(\mathbf{x})) 
		\end{align*}
		if $V_{t_{\min}}(\widetilde{Z}_u(\mathbf{x}))$ is small enough.
		Similarly, since $V_0(\widetilde{Z}_u(\mathbf{x}))=1$ as $\widetilde{Z}_u(\mathbf{x})$ is convex, we have for $t_{\min}=0$,
		\begin{align*}
			\lvert\varphi(\widetilde{Z}_u(\mathbf{x}))\rvert
			&\geq V_{t_{\min}}(\widetilde{Z}_u(\mathbf{x}))\Big(\alpha_{t_{\min}}-\sum_{i\in\{0,\dots,d\}:\alpha_i<0}\lvert\alpha_i\rvert V_i(\widetilde{Z}_u(\mathbf{x}))\Big)\\
			&\geq V_{t_{\min}}(\widetilde{Z}_u(\mathbf{x}))\Big(\alpha_{t_{\min}}-\sum_{i\in\{0,\dots,d\}:\alpha_i<0}\lvert\alpha_i\rvert C(1,i) V_1(\widetilde{Z}_u(\mathbf{x}))^i\Big)\\
			&\geq \frac{\alpha_{t_{\min}}}{2} V_{t_{\min}}(\widetilde{Z}_u(\mathbf{x})) 
		\end{align*}
		if $V_1(\widetilde{Z}_u(\mathbf{x}))$ is small enough.
		In Case 2 it additionally holds 
		\begin{align*}
			\varphi\Big(K\cap\bigcap_{j=1}^{k-1} (U_{m_{i_j}}+x_{i_j})\Big)=	\varphi\Big(\bigcap_{j=1}^{k-1} (U_{m_{i_j}}+x_{i_j})\Big)=0
		\end{align*}
		for $(x_1,\dots,x_{k})\in S_{m_1,\dots,m_{k}}$, $(m_1,\dots,m_{k})\in M$, $i_1,\dots,i_{k-1}\in\{1,\dots,k\}$ with $i_j\neq i_\ell$ for $j\neq \ell$ and $\hat{K}_{i}\subseteq W_r$ for $i\in\{1,\dots,k\}$. 	Note that we can choose the parameter $R>0$ in the definition of $S_{m_1,\dots,m_k}$ by Lemma \ref{lemma:geometric lemma} sufficiently small such that all estimations above hold while \eqref{eq:prop2_Sm} remains valid. Then, if $\eta(S_K)=0$, it holds
		\begin{align*}
			\lvert D_{\hat{x}_1,\dots,\hat{x}_{k}}^{k}\varphi(Z_u\cap K)\rvert=\Big\lvert	\varphi\Big(Z_u\Big(\eta+\sum_{i=1}^{k}\delta_{\hat{x}_i}\Big)\cap K\Big)\Big\rvert\geq 	\frac{\alpha_{t_{\min}}}{2} V_{t_{\min}}(\widetilde{Z}_u(\mathbf{x})) =: C_{\hat{x}_1,\dots,\hat{x}_{k}}>0
		\end{align*}
		for $(x_1,\dots,x_{k})\in S_{m_1,\dots,m_{k}}$ with $\hat{K}_{i}\subseteq W_r$ for $i\in\{1,\dots,k\}$ and $(m_1,\dots,m_{k})\in M$, i.e.\ we have

		\begin{align}\label{eq:prob_lower_var}
			\p( \lvert D_{\hat{x}_1,\dots,\hat{x}_{k}}^{k}\varphi(Z_u\cap K)\rvert\geq C_{\hat{x}_1,\dots,\hat{x}_{k}})
			\geq \p(\eta(S_{K})=0) \geq e^{-\lambda(S_{K_{m_1}})}>0.
		\end{align}
		
		For $\mathbf{m}=(m_1,\dots,m_{k})\in M$ let $h_{\mathbf{m}}:\mathbb{R}^d\to\mathbb{R}_{\geq 0}$ be defined as
		\begin{align*}
			h_{\mathbf{m}}(x_1,\dots,x_{k})=C_{\hat{x}_1,\dots,\hat{x}_{k}}^2\mathbbm{1}\{(x_1,\dots,x_{k})\in S_{m_1,\dots,m_{k}}\}e^{-\lambda(S_{K_{m_1}})}.
		\end{align*}
		The translation invariance of the intrinsic volumes provides
		\begin{align*}
			C_{\hat{x}_1,\dots,\hat{x}_{k}}=C_{(x_1+z,m_1),\dots,(x_{k}+z, m_{k})}
		\end{align*}	
		for any $z\in\mathbb{R}^d$. Then, by \eqref{eq:prop1_Sm}, $h_{\mathbf{m}}$ is translation invariant and with \eqref{eq:prop2_Sm} and \eqref{eq:prop3_Sm} we get
		\begin{align}\label{eq:int_lower_var}
			\int_{(\mathbb{R}^d)^{k-1}}\mathbbm{1}\Big\{K_{m_{k}}\cap\bigcap_{i=1}^{k-1}(K_{m_i}+y_i)\neq\emptyset\Big\}h_{\mathbf{m}}(y_1,\dots,y_{k-1},\mathbf{0})\;\mathrm{d}(y_1,\dots,y_{k-1})>0.
		\end{align}
		Moreover, for $\widehat{W}_{r}=\{(x_1,\dots,x_{k})\in(\mathbb{R}^d)^{k}:\hat{K}_i\subseteq W_r, i\in\{1,\dots,k\}\}$ and $y_1,\dots,y_{k-1}\in\mathbb{R}^d$ satisfying  $K_{m_{k}}\cap\bigcap_{i=1}^{k-1}(K_{m_i}+y_i)\neq\emptyset$, there exists $r_0(m_1,\dots,m_{k})>0$ only depending on $m_1,\dots,m_{k}$ such that
		\begin{align}
			&\lambda_d(\{y\in\mathbb{R}^d:(y_1+y,\dots,y_{k-1}+y,y)\in\widehat{W}_{r}\})\nonumber\\&\geq \lambda_d\Big(\Big\{y\in W_r:\mathrm{d}(y,\partial W_r)\geq R(K_{m_{k}})+2\max_{i\in\{1,\dots,k-1\}}R(K_{m_i})\Big\}\Big)\geq\frac{\lambda_d(W_r)}{2}
			\label{eq:wdach_lower_var}
		\end{align}
		for all $r\geq r_0(m_1,\dots,m_{k})$.
		Now, let $r_0>0$ be large enough such that $\Q(\{m\in M:r_0(m_1,\dots,m_{k})\leq r_0\})>0$. Then we have with \eqref{eq:D^n}, \eqref{eq:prob_lower_var}, \eqref{eq:int_lower_var}, \eqref{eq:wdach_lower_var} and the translation invariance of $h_{\mathbf{m}}$ for $r\geq r_0$,
		\allowdisplaybreaks
		\begin{align*}
			&\E\idotsint \lvert D_{\hat{x}_1,\dots,\hat{x}_{k}}^{k}\varphi(Z_u\cap W_r)\rvert^2\;\lambda^{k}(\mathrm{d}(\hat{x}_1,\dots,\hat{x}_{k}))\\
			&\geq \idotsint\p\Big(\Big\lvert D_{\hat{x}_1,\dots,\hat{x}_{k}}^{k}\varphi\Big(Z_u\cap W_r\cap \bigcap_{i=1}^{k}\hat{K}_i\Big)\Big\rvert\geq C_{\hat{x}_1,\dots,\hat{x}_{{k}}}\Big)C_{\hat{x}_1,\dots,\hat{x}_{{k}}}^2\;\lambda^{k}(\mathrm{d}(\hat{x}_1,\dots,\hat{x}_{k}))\\
			&\geq \idotsint \mathbbm{1}\{(m_1,\dots,m_{k})\in M,(x_1,\dots,x_{k})\in\widehat{W}_{r}\} h_{\mathbf{m}}(x_1,\dots,x_{k}) \;\lambda^{k}(\mathrm{d}(\hat{x}_1,\dots,\hat{x}_{k}))\\
			&\geq \gamma^{k}\int_{M}\int_{(\mathbb{R}^d)^{k}}\mathbbm{1}\{(y_1+y_{k},\dots,y_{k-1}+y_{k},y_{k})\in\widehat{W}_{r}\}\mathbbm{1}\Big\{K_{m_{k}}\cap\bigcap_{i=1}^{k-1}(K_{m_i}+y_i)\neq\emptyset\Big\} \\&\hspace*{3cm}\times h_{\mathbf{m}}(y_1,\dots,y_{k-1},\mathbf{0}) \;\mathrm{d}(y_1,\dots,y_{k})\;\mathbb{Q}^{k}(\mathrm{d}(m_1,\dots,m_{k}))\\
			&\geq \gamma^{k}\int_{M}\int_{(\mathbb{R}^d)^{k-1}}\hspace*{-0.3cm} \lambda_d(\{y\in\mathbb{R}^d:(y_1+y,\dots,y_{k-1}+y,y)\in\widehat{W}_{r}\})\mathbbm{1}\Big\{K_{m_{k}}\cap\bigcap_{i=1}^{k-1}(K_{m_i}+y_i)\neq\emptyset\Big\} \\&\hspace*{3cm}\times h_{\mathbf{m}}(y_1,\dots,y_{k-1},\mathbf{0}) \;\mathrm{d}(y_1,\dots,y_{k-1})\;\mathbb{Q}^{k}(\mathrm{d}(m_1,\dots,m_{k}))\\
			&\geq c_4\lambda_d(W_r)
		\end{align*}
		for a suitable constant $c_4>0$, which completes the proof.	
	\end{proof}
	\subsection{Central limit theorem}
	For the proof of the qualitative central limit theorem we show that $\gamma_1,\gamma_2$ and $\widetilde{\gamma}_3$ from Theorem \ref{thm:clt_theory_Wasserstein} vanish for $r\to\infty$. For the proof of the quantitative central limit theorems we show that $\gamma_1,\dots,\gamma_6$ from Theorem \ref{thm:clt_theory_Wasserstein} and Theorem \ref{thm:clt_kolmogorov_theory} are of the right order under the additional moment assumptions. To this end we mainly use \eqref{eq:E[D^n]}. 
	\begin{proof}[Proof of Theorem \ref{thm:clt}] 
		We consider $F=\frac{ \varphi(Z_u\cap W_r)-\E[ \varphi(Z_u\cap W_r)]}{\sqrt{\Var[ \varphi(Z_u\cap W_r)]}}$. For this standardised random variable it holds that
		\begin{align*}
			D_{\hat{x}_1,\dots,\hat{x}_n}^nF=\frac{D_{\hat{x}_1,\dots,\hat{x}_n}^n \varphi(Z_u\cap W_r)}{\sqrt{\Var[ \varphi(Z_u\cap W_r)]}}.
		\end{align*}
		To show $a)$ we start with bounding $\gamma_1,\gamma_2$ and $\widetilde{\gamma}_3$ from Theorem \ref{thm:clt_theory_Wasserstein}. 
		Let $f\big(K_1,\dots,K_t\big)=V_d\big(\bigcap_{i=1}^tK_i^{\sqrt{d}}\big)$ for $K_i\in\mathcal{K}^d$, $i\in\{1,\dots,t\}$ and $t\in\mathbb{N}$. Then, for $\hat{x}_1=(x_1,m_1),\hat{x}_2=(x_2,m_2),\hat{x}_3=(x_3,m_3)\in\mathbb{R}^d\times\MM$ we have with \eqref{eq:E[D^n]} for suitable constants $c_1,c_2>0$,
		\begin{align*}
			\E[(D_{\hat{x}_1}F)^2(D_{\hat{x}_2}F)^2]&\leq \frac{c_1}{\Var[ \varphi(Z_u\cap W_r)]^2}f(\hat{K}_1, W_r)^2f(\hat{K}_2, W_r)^2
		\end{align*}
		and
		\begin{align*}
			&\E[(D_{\hat{x}_1,\hat{x}_3}^2F)^2(D_{\hat{x}_2,\hat{x}_3}^2F)^2]\leq \frac{c_2}{\Var[ \varphi(Z_u\cap W_r)]^2}f(\hat{K}_1, \hat{K}_3, W_r)^2f(\hat{K}_2, \hat{K}_3, W_r)^2.
		\end{align*}
		Since with Equation \eqref{eq:int_V_d A_1+xcapA_2} it holds for $j\in\{1,2\}$ that
		\begin{align*}
			\int_{\mathbb{R}^d}f(\hat{K}_j, W_r)f(\hat{K}_j,\hat{K}_3, W_r)\;\mathrm{d}x_j&\leq\nonumber f(K_{m_j})\int_{\mathbb{R}^d}f(\hat{K}_j,\hat{K}_3, W_r)\;\mathrm{d}x_j= f(K_{m_j})^2f(\hat{K}_3, W_r),
		\end{align*}
		we get for $r\geq 1$,
		\begin{align*}
			\gamma_1^2&\leq \frac{4\sqrt{c_1c_2}\gamma^3}{\Var[ \varphi(Z_u\cap W_r)]^2}\int_{\MM^3}f(K_{m_1})^2f(K_{m_2})^2\int_{\mathbb{R}^d}f(\hat{K}_3, W_r)^2\;\mathrm{d}x_3\;\mathbb{Q}^3(\mathrm{d}(m_1,m_2,m_3))\\
			&\leq \frac{4\sqrt{c_1c_2}\gamma^3}{\Var[\varphi(Z_u\cap W_r)]^2} \int_{\MM^3}f(K_{m_1})^2f(K_{m_2})^2f(K_{m_3})^2f(W_r)\;\mathbb{Q}^3(\mathrm{d}(m_1,m_2,m_3))\\
			&\leq  \frac{4\sqrt{c_1c_2}f(W_r)\gamma^3}{\Var[ \varphi(Z_u\cap W_r)]^2}\Big(\int_{\MM}f(K_{m})^2\;\mathbb{Q}(\mathrm{d}m)\Big)^3.
		\end{align*}
		Note that the integral exists because of the second order moment assumption in \eqref{Assumption_fourth_moment}. By Theorem \ref{thm:variance_limit} it holds $\Var[\varphi(Z_u\cap W_r)]\geq \frac{\sigma_0}{2}V_d(W_r)= \frac{\sigma_0r^d}{2}V_d(W)$ for $r$ large enough. Then, since $f(W_r)\leq r^df(W)$ for $r\geq 1$ we get for $r$ large enough,
		\begin{align}\label{gamma1}
			\gamma_1^2\leq \frac{c_3}{V_d(W_r)}
		\end{align}
		for a suitable constant $c_3>0$ and hence $\gamma_1\to 0$ as $r\to\infty$.
		Analogously, we get for $\gamma_2$,
		\begin{align}\label{gamma2}
			\gamma_2^2\leq  \frac{c_2f(W_r)\gamma^3}{\Var[ \varphi(Z_u\cap W_r)]^2}\Big(\int_{\MM}f(K_{m})^2\;\mathbb{Q}(\mathrm{d}m)\Big)^3
			\leq \frac{c_4}{V_d(W_r)}
		\end{align}
		for a suitable constant $c_4>0$ and $r$ large enough. Thus, $\gamma_2\to 0$ as $r\to\infty$.
		For the bound of $\widetilde{\gamma}_3$ we use that
		\begin{align}\label{eq:gamma_tilde}
			\E[\lvert D_{\hat{x}_1}F\rvert ^3]\leq \frac{c_5}{\Var[ \varphi(Z_u\cap W_r)]^{3/2}}f(\hat{K}_1, W_r)^3
		\end{align}
		from \eqref{eq:E[D^n]} for some $c_5>0$. Since additionally $\E[\lvert D_{\hat{x}_1}F\rvert^{3/2}]^{2/3}\leq \E[\lvert D_{\hat{x}_1}F\rvert^{3}]^{1/3}$ by Jensen's inequality, we have for fixed $t>0$,
		\begin{align*}
			\widetilde{\gamma}_3&\leq \gamma \int_{\MM}\int_{\mathbb{R}^d}\mathbbm{1}\{f(K_{m_1})\leq t\}	\E[\lvert D_{\hat{x}_1}F\rvert ^3] \;\mathrm{d}x_1\;\mathbb{Q}(\mathrm{d}m_1)\\&\quad+2\gamma\int_{\MM}\int_{\mathbb{R}^d}\mathbbm{1}\{f(K_{m_1})> t\}	\E[\lvert D_{\hat{x}_1}F\rvert ^3]^{2/3} \;\mathrm{d}x_1\;\mathbb{Q}(\mathrm{d}m_1)=:I_1(t,r)+I_2(t,r).
		\end{align*}
		For $I_1$ we have with \eqref{eq:gamma_tilde},
		\begin{align*}
			I_1(t,r)&\leq \frac{c_5}{\Var[ \varphi(Z_u\cap W_r)]^{3/2}}\int_\MM\mathbbm{1}\{f(K_{m_1})\leq t\}f(K_{m_1})^2\int_{\mathbb{R}^d}f(\hat{K}_1, W_r)\;\mathrm{d}x_1\;\mathbb{Q}(\mathrm{d}m_1)\\
			&\leq \frac{tc_5f(W_r)}{\Var[ \varphi(Z_u\cap W_r)]^{3/2}}\int_\MM\mathbbm{1}\{f(K_{m_1})\leq t\}f(K_{m_1})^2\;\mathbb{Q}(\mathrm{d}m_1)\leq \frac{c_6t}{\sqrt{V_d(W_r)}}
		\end{align*} 
		for a suitable constant $c_6>0$ if $r$ is large enough and hence, $I_1(t,r)\to 0$ as $r\to\infty$ for any $t>0$. Similarly, we get for $I_2$,
		\begin{align*}
			I_2(t,r)&\leq \frac{2c_5^{2/3}}{\Var[ \varphi(Z_u\cap W_r)]}\int_\MM\mathbbm{1}\{f(K_{m_1})> t\}f(K_{m_1})\int_{\mathbb{R}^d}f(\hat{K}_1, W_r)\;\mathrm{d}x_1\;\mathbb{Q}(\mathrm{d}m_1)\\
			&= \frac{2c_5^{2/3}f(W_r)}{\Var[ \varphi(Z_u\cap W_r)]}\int_\MM\mathbbm{1}\{f(K_{m_1})> t\}f(K_{m_1})^2\;\mathbb{Q}(\mathrm{d}m_1)\\
			&\leq c_7\int_\MM\mathbbm{1}\{f(K_{m_1})> t\}f(K_{m_1})^2\;\mathbb{Q}(\mathrm{d}m_1)
		\end{align*}
		for a suitable constant $c_7>0$ if $r$ is large enough. Note that by the second order moment assumption, $\int_\MM\mathbbm{1}\{f(K_{m_1})> t\}f(K_{m_1})^2\;\mathbb{Q}(\mathrm{d}m_1)\to0$ as $t\to\infty$. This means that for any $\varepsilon>0$ we can choose $\hat{t}>0$ such that $I_2(\hat{t},r)\leq \frac{\varepsilon}{2}$ for $r$ sufficiently large. Then, for $r$ large enough such that also $I_1(\hat{t},r)\leq\frac{\varepsilon}{2}$, we have $\widetilde{\gamma}_3<\varepsilon$. Altogether, we have shown that $\gamma_1,\gamma_2$ and $\widetilde{\gamma}_3$ vanish as $r\to\infty$, which provides $a)$.
		
		For the proof of the Wasserstein distance in $b)$ it is sufficient to additionally bound $\gamma_3$ from Theorem \ref{thm:clt_theory_Wasserstein} since we have already shown in \eqref{gamma1} and \eqref{gamma2} that $\gamma_1$ and $\gamma_2$ are of the right order. To this end we use \eqref{eq:gamma_tilde}. Then, with Equation \eqref{eq:int_V_d A_1+xcapA_2} we have for $r$ large enough,
		\begin{align*}
			\gamma_3&\leq \frac{c_5\gamma}{\Var[ \varphi(Z_u\cap W_r)]^{3/2}} \int_\MM f(K_{m_1})^2\int_{\mathbb{R}^d}f(\hat{K}_1, W_r)\;\mathrm{d}x_1\;\mathbb{Q}(\mathrm{d}m_1)\\
			&= \frac{c_5\gamma f(W_r)}{\Var[ \varphi(Z_u\cap W_r)]^{3/2}} \int_\MM f(K_{m_1})^3\;\mathbb{Q}(\mathrm{d}m_1)\\
			&\leq \frac{c_8}{\sqrt{V_d(W_r)}} 
		\end{align*}
		for a suitable constant $c_8>0$. Note that the integral is finite by the moment assumption \eqref{Assumption_fourth_moment} of order three, which provides together with \eqref{gamma1} and \eqref{gamma2} the result for the Wasserstein distance.
		
		To show $c)$ we additionally bound $\gamma_4,\gamma_5$ and $\gamma_6$ from Theorem \ref{thm:clt_kolmogorov_theory}.
		From \cite[Lemma 4.2]{LPS16} we know that
		\begin{align}\label{eq:E[f4]}
			\E[F^4]\leq\max\Big\{256\Big(\int(\E[(D_{\hat{x}}F)^4]^{1/2})\;\lambda(\mathrm{d}\hat{x})\Big)^2,4\int\E[(D_{\hat{x}}F)^4]\;\lambda(\mathrm{d}\hat{x})+2\Big\}.
		\end{align}
		The bound
		\begin{align}
			\E[\lvert D_{\hat{x}_1}F\rvert ^4]\leq \frac{c_9}{\Var[ \varphi(Z_u\cap W_r)]^2}f(\hat{K}_1, W_r)^4\label{eq:E[dx4]}
		\end{align}
		from Lemma \ref{lemma:diff_operator} for some constant $c_9>0$ provides
		\begin{align*}
			&\Big(\gamma\int_\MM\int_{\mathbb{R}^d}\E[(D_{\hat{x}}F)^4]^{1/2}\;\mathrm{d}x_1\;\mathbb{Q}(\mathrm{d}m_1)\Big)^2\nonumber\\&\leq\frac{ c_9\gamma^2}{\Var[ \varphi(Z_u\cap W_r)]^2}\Big(\int_\MM f(K_{m_1})\int_{\mathbb{R}^d}f(\hat{K}_1, W_r)\;\mathrm{d}x_1\;\mathbb{Q}(\mathrm{d}m_1)\Big)^2\nonumber\\
			&= \frac{ c_9\gamma^2 f(W_r)^2}{\Var[ \varphi(Z_u\cap W_r)]^2}\Big(\int_\MM f(K_{m_1})^2\;\mathbb{Q}(\mathrm{d}m_1)\Big)^2
		\end{align*}
		and
		\begin{align}
			&\gamma\int_\MM\int_{\mathbb{R}^d}\E[(D_{\hat{x}}F)^4]\;\mathrm{d}x_1\;\mathbb{Q}(\mathrm{d}m_1)\nonumber\\&\leq \frac{c_9\gamma}{\Var[ \varphi(Z_u\cap W_r)]^2}	\int_\MM f(K_{m_1})^3\int_{\mathbb{R}^d}f(\hat{K}_1, W_r)\;\mathrm{d}x_1\;\mathbb{Q}(\mathrm{d}m_1)\nonumber\\
			&= \frac{ c_9\gamma f( W_r)}{\Var[ \varphi(Z_u\cap W_r)]^2}	\int_\MM f(K_{m_1})^4\;\mathbb{Q}(\mathrm{d}m_1).\label{eq:E[dxf4]}
		\end{align}
		Since the variance and $f(W_r)$ grow with order $r^d$, this bound leads with the fourth moment assumption and Equation \eqref{eq:E[f4]} to $\E[F^4]\leq c_{10}$ for a suitable constant $c_{10}>0$ and $r$ large enough. Hence, with Equation \eqref{eq:E[D^n]} we have
		\begin{align*}
			\gamma_4&\leq\frac{\gamma c_{10}^{1/4}}{2}\int_\MM\int_{\mathbb{R}^d}(\E[(D_{\hat{x}_1}F)^4])^{3/4}\;\mathrm{d}x_1\;\mathbb{Q}(\mathrm{d}m_1)\\
			&\leq \frac{c_{11}}{\Var[ \varphi(Z_u\cap W_r)]^{3/2}} \int_\MM f(K_{m_1})^2\int_{\mathbb{R}^d}f(\hat{K}_1, W_r)\;\mathrm{d}x_1\;\mathbb{Q}(\mathrm{d}m_1)\\
			&=\frac{c_{11}f(W_r)}{\Var[ \varphi(Z_u\cap W_r)]^{3/2}}\int_\MM f(K_{m_1})^3\;\mathbb{Q}(\mathrm{d}m_1)\leq \frac{c_{12}}{\sqrt{V_d(W_r)}}
		\end{align*}
		for suitable constants $c_{11},c_{12}>0$ and $r$ large enough.
		By \eqref{eq:E[dxf4]} we get for $\gamma_5$,
		\begin{align*}
			\gamma_5^2\leq \frac{c_{13}}{V_d(W_r)}
		\end{align*}
		for $r$ large enough and some constant $c_{13}>0$. Finally, for $\gamma_6$ we need the estimate
		\begin{align*}
			\E[(D_{\hat{x}_1,\hat{x}_2}^2F)^4]&\leq \frac{c_{14}}{\Var[ \varphi(Z_u\cap W_r)]^2}f(\hat{K}_1, \hat{K}_2, W_r)^4\\&\leq  \frac{c_{14}}{\Var[ \varphi(Z_u\cap W_r)]^2}f(\hat{K}_1, \hat{K}_2, W_r)^2f(\hat{K}_1, W_r)^2.
		\end{align*}
		for some $c_{14}>0$.
		This provides with \eqref{eq:E[dx4]} and $c_{15}=6\sqrt{c_7c_{14}}+3c_{14}$,
		\begin{align*}
			\gamma_6^2&\leq \frac{c_{15}\gamma^2}{\Var[ \varphi(Z_u\cap W_r)]^2} \int_{\MM^2}\int_{(\mathbb{R}^d)^2}f(\hat{K}_1, W_r)^2f(\hat{K}_1,\hat{K}_2, W_r)^2\;\mathrm{d}(x_1,x_2)\;\mathbb{Q}^2(\mathrm{d}(m_1,m_2))\\
			&\leq \frac{c_{15}\gamma^2}{\Var[ \varphi(Z_u\cap W_r)]^2}\int_{\MM^2}f(K_{m_1})^2f(K_{m_2})\int_{(\mathbb{R}^d)^2}f(\hat{K}_1,\hat{K}_2, W_r)\;\mathrm{d}(x_1,x_2)\;\mathbb{Q}^2(\mathrm{d}(m_1,m_2))\\
			&\leq \frac{c_{15}\gamma^2f(W_r)}{\Var[ \varphi(Z_u\cap W_r)]^2}\int_{\MM}f(K_{m_1})^3\;\mathbb{Q}(\mathrm{d}m_1)\int_{\MM}f(K_{m_2})^2\;\mathbb{Q}(\mathrm{d}m_2)\\
			&\leq \frac{c_{16}}{V_d(W_r)}
		\end{align*}
		for a suitable constant $c_{16}>0$ and $r$ large enough.
		Altogether, the estimates for $\gamma_1,\dots,\gamma_6$ complete the proof of the quantitative central limit theorem in Kolmogorov distance.
	\end{proof}
	\appendix
	\section{Appendix: A new lower variance bound for Poisson functionals}
	\label{app:lower_variance}
	This appendix gives the proof of Theorem \ref{thm:lowervarbound}, the generalised reverse Poincar\'e inequality. For $k=1$ this theorem was already proven in \cite[Theorem 1.1]{ST22}. The proof for general $k$ is similarly to the proof for $k=1$ based on the Fock space representation.
	\begin{proof}[Proof of Theorem \ref{thm:lowervarbound}.]
		Let $x_1,\dots,x_k\in\mathbb{X}$.
		The Fock space representation \eqref{eq:fockspace} provides for $f_m(x_1,\dots,x_m)=\frac{1}{m!}\E[D_{x_1,\dots,x_m}^kF]$ and $f_0=\E[F]$,
		\begin{align*}
			\E[F^2]&=\E[F]^2+\sum_{m=1}^{\infty}m!\lVert f_m\rVert_m^2=\sum_{m=0}^{\infty}m!\lVert f_m\rVert_m^2
		\end{align*}
		and for $j\in\mathbb{N}$ together with Fubini's theorem and the monotone convergence theorem,
		\begin{align*}
			&\E\int_{\mathbb{X}^j}(D_{x_1,\dots,x_{j}}^{j}F)^2\;\mu^{j}(\mathrm{d}(x_1,\dots,x_{j}))\\&=\int_{\mathbb{X}^j} \sum_{m=0}^{\infty}\frac{1}{m!}\int_{\mathbb{X}^m}\E[D^m_{y_1,\dots,y_m}(D^j_{x_1,\dots,x_j}F)]^2\;\mu^{m}(\mathrm{d}(y_1,\dots,y_{m}))\;\mu^{j}(\mathrm{d}(x_1,\dots,x_{j}))\\&
			= \sum_{m=0}^{\infty}\frac{1}{m!}\int_{\mathbb{X}^{m+j}} \E[D^{m+j}_{x_1,\dots,x_{m+j}}F]^2\;\mu^{m+j}(\mathrm{d}(x_1,\dots,x_{m+j}))
			\\&=  \sum_{m=j}^{\infty}\frac{\prod_{i=0}^{j-1}(m-i)}{m!}\int_{\mathbb{X}^{m}} \E[D^{m}_{x_1,\dots,x_{m}}F]^2\;\mu^{m}(\mathrm{d}(x_1,\dots,x_{m}))
			\\&=
			\sum_{m=1}^{\infty}\prod_{i=0}^{j-1}(m-i)m!\lVert f_m\rVert_m^2.
		\end{align*}
		This means, Assumption \eqref{Assumption_variance_theoretisch} is equivalent to
		\begin{align}\label{assumption_fock_space}
			\sum_{m=1}^{\infty}\prod_{i=0}^{k-1}(m-i)m!\lVert f_m\rVert_m^2(\alpha-m+k)\geq 0.
		\end{align}
		Now, choose $c(\alpha,k)\geq\prod_{i=0}^{k-1}(m-i)(\alpha-m+k+1)=:g(m)$ for all $m\in\mathbb{N}_0$, which is possible since $g:\mathbb{N}_0\to\mathbb{R}$ is uniformly bounded in $m$ as the leading polynomial occurs with a negative sign. Then,
		\begin{align*}			&c(\alpha,k)\Var[F]-\E\int(D_{x_1,\dots,x_{k}}^{k}F)^2\;\mu^k(\mathrm{d}(x_1,\dots,x_k))\\&=\sum_{m=1}^{\infty}m!\lVert f_m\rVert_m^2\Big(c(\alpha,k)-\prod_{i=0}^{k-1}(m-i)\Big)\\&\geq \sum_{m=1}^{\infty}m!\lVert f_m\rVert_m^2\prod_{i=0}^{k-1}(m-i)(\alpha-m+k)\geq 0
		\end{align*}
		by \eqref{assumption_fock_space}, which completes the proof.
	\end{proof}
	
	\subsection*{Acknowledgements}
	The author would like to thank Matthias Schulte for drawing her attention to the field of Poisson shot noise processes and for many valuable discussions on geometric functionals of excursion sets. Additionally, the author would like to thank Matthias Lienau for some helpful comments on this paper.

\end{document}